\documentclass[onethmnum,onefignum,final]{siamltex}
\usepackage{amsmath,amssymb,placeins}

\usepackage{graphicx,color}
\usepackage{MnSymbol}

\usepackage{showkeys}

\newcommand{\lip}[3]{\left(#1,#2\right)_{#3}}

\newcommand{\sj}[1]{\left[ #1 \right]}

\newcommand{\sa}[1]{\left\{ #1 \right\}}

\newcommand{\tj}[1]{\lsem #1 \rsem}
\newcommand{\dspace}[1]{S_{n}^{h,#1}}%(\Omega \times I_n)}
\newcommand{\dspaceT}[1]{V^{h,#1}}%(\Omega \times [0,T])}
\newcommand{\Tspace}{S_{n,\text{Trefftz}}^{h,p}}
\newcommand{\TspaceT}{V_{\text{Trefftz}}^{h,p}}

\newcommand{\dgnorm}[1]{|\kern-.25mm |\kern-.25mm|{#1}| \kern -.25mm | \kern -.25mm |}
\newcommand{\dgnorml}[1]{|\kern-.25mm |\kern-.25mm|{#1}| \kern -.25mm | \kern -.25mm |_{\star}}

\newtheorem{remark}[theorem]{Remark}
\newtheorem{assumption}[theorem]{Assumption}
\newtheorem{example}{Example}

\newcommand{\ltwo}[2]{\|{#1}\|_{#2}}
\newcommand{\nv}{\mathbf{n}}

\newcommand{\ud}{\,\mathrm{d}}
\newcommand{\ha}{\frac{1}{2}}

\newcommand{\babla}{\widetilde{\nabla}}
\newcommand{\hGamman}{\hat{\Gamma}_n}
\newcommand{\binit}{b^{\text{init}}}

\newcommand{\solnspace}{\mathcal{X}}

\DeclareMathOperator{\diam}{diam}

\newcommand{\reply}[1]{{#1}} 
\newcommand{\our}[1]{{#1}}

\title{A Trefftz polynomial \\ space-time discontinuous Galerkin method \\ for the second order wave equation}
\author{Lehel Banjai\thanks{Maxwell Institute for Mathematical Sciences, School of Mathematical \& Computer Sciences, Heriot-Watt University, Edinburgh EH14 4AS, UK.
({\tt l.banjai@hw.ac.uk})} 
\and Emmanuil H.~Georgoulis\thanks{
Department of  Mathematics,
University of Leicester,
University Road,
Leicester , LE1 7RH,
United Kingdom, 
and Department of Mathematics, School of Applied Mathematical and Physical Sciences, National Technical University of Athens, Zografou 15780, Greece. ({\tt Emmanuil.Georgoulis@le.ac.uk})} \and Oluwaseun Lijoka\thanks{School of Mathematical \& Computer Sciences, Heriot-Watt University, Edinburgh EH14 4AS, UK.
({\tt ofl1@hw.ac.uk})}
}
\begin{document}
%\slugger{sinum}{xxxx}{xx}{x}{x--x}
\maketitle

\begin{abstract}
A new space-time discontinuous Galerkin (dG) method\footnote{\small A brief report on the results of this paper was published in the proceedings paper \cite{dgwave_waves}.} utilising special Trefftz polynomial basis functions is proposed and fully analysed for the scalar wave equation in a second order formulation. The dG method considered is motivated by the class of interior penalty dG methods, as well as by the classical work of Hughes and Hulbert \cite{HulH88,HulH90}. The choice of the penalty terms included in the bilinear form is essential for both the theoretical analysis and for the practical behaviour of the method for the case of lowest order basis functions. A best approximation result is proven for this new space-time dG method with Trefftz-type basis functions. Rates of convergence are proved in any dimension and verified numerically in spatial dimensions $d = 1$ and $d = 2$.  Numerical experiments highlight the effectivness of the Trefftz method in problems with energy at high frequencies.
\end{abstract}

\section{Introduction}
The use of non-polynomial basis functions in the context of finite element methods dates back at least to the early 1980s \cite{MR806382}.
In recent years, there has been much interest in using non-polynomial basis functions to discretize wave propagation problems in the frequency domain \cite{CesD98,MoiHP}.
A prominent example is the use of plane wave bases  to solve the Helmholtz equation at high frequencies. The motivation behind this approach is to reduce the number of degrees of freedom per wavelength required to obtain accurate results. Thus obtained Trefftz methods have proved very successful in practice, hence  it is a natural question to ask whether they can be extended and whether they can be equally successful in the time-domain. 

The most natural way of including space-time Trefftz basis functions is within the confines of a space-time discontinuous Galerkin (dG) method. 
In this work, we construct and analyse a new space-time interior penalty dG method for the second order wave equation, that can utilize Trefftz polynomials as local basis functions. The method discretizes the wave equation in primal form and is defined using space-time slabs to ensure solvability on each time-step, as well as to aid the presentation and the analysis. However, with minor modifications, completely unstructured space-time meshes are, in principle, possible in the proposed space-time dG framework. This results in a stable, dissipative scheme for general polynomial bases. For Trefftz basis, we prove quasi-optimality in the dG energy norm, which we show to be an upper bound for a standard space-time energy norm. Numerical results in the dG norm show that the \reply{theoretical estimates of the convergence order are optimal}. Furthermore, the numerical results show that higher order schemes have excellent energy conservation properties and work well for systems with energy at high frequencies. Comparison with standard polynomial bases show that the same convergence order and approximation properties is obtained with considerably fewer degrees of freedom. Furthermore, the implementation is less expensive due to integration being restricted to the space-time skeleton.

% The numerical solution of time-domain wave problems posses a number of interesting challenges, especially in the context of high-order discretisations in both space and in time. 

% Typically, finite element methods for the linear (and some spatially nonlinear) wave problems are based on finite element discretisations of the spatial variables complemented with standard time-stepping schemes, the most popular of which are explicit, such as the leap-frog scheme, due to the acceptable CFL restrictions. 

Space-time variational methods for wave-type problems have appeared in the late 1980's with the works of Hughes and Hulbert \cite{HulH88,HulH90} and Johnson \cite{Joh93}.
First numerical experiments with  Trefftz space-time dG methods were performed in \cite{PetFT}. Currently there is significant activity on the topic \cite{Egger1,KreMPS,Egger2,moiola_proc,KreSTW}.  \reply{ In particular in \cite{Egger1,KreMPS} a Trefftz space-time local dG method for the Maxwell equations, written as a first order system, resulting in a two-field formulation, has been analysed. The dG method designed and analyzed in the present work involves interior-penalty-type numerical fluxes in space, resulting in one-field approximation in space, as well as a one-field dG method in time, with very similar computational stencil widths compared to the local dG method introduced in \cite{KreMPS}.}

%The main differences between the two works are different formulations resulting in differeing analysis and our use of a second order form of the wave equation, that doesn't requiring the use of scalar unknowns.

Typically,  finite element methods for linear  (and some spatially nonlinear) wave problems are based on a continuous or discontinuous finite element discretisation of the spatial variables complemented with standard time-stepping schemes, the most popular of which are explicit, such as the leap-frog scheme, due to the acceptable CFL restrictions. Though even the low order methods can be conservative, for acceptable accuracy when energies at high frequencies are excited, higher order methods are  essential \cite{CohBook,AinMM,AguDE}. Compared to these methods, introduction of higher order approximations is much more straightforward in the context of space-time dG methods. \reply{Moreover, space-time dG methods, such as the one presented below, do \emph{not} require any CFL-type restrictions, owing to their implicit time-stepping interpretation.}  If a general space-time mesh can be used, a judicious choice of the mesh can result in a quasi-explicit method where only small local systems need to be solved \cite{FalR,MonR}. %In conclusion we expect that the space-time Trefftz methods can be competitive 

The remainder of this work is structured as follows. In the next section, we introduce the model problem. In Sections~\ref{sec:space} and \ref{sec:method} we construct the space-time interior penalty method and we show its stability. We proceed in Section~\ref{sec:Tspace} to analyse polynomial Trefftz spaces and prove quasi-optimality. Finally in Section~\ref{apriori}, we prove convergence rates for the $d$-dimensional method in space, $d=1,2,3$; moreover, we also provide $hp$-version a priori bounds for $d=1$. A series of numerical experiments in Section \ref{sec:numerics}, illustrates the theoretical findings and highlights the good performance of the proposed method in practice.

\section{Model problem}\label{sec:model}

We consider the wave equation
\begin{equation}
  \label{eq:wave}  
\begin{aligned}
 \ddot u -\nabla \cdot \left(a\, \nabla u\right) &= 0 & &\text{in } \Omega \times [0,T], \\
 u &= 0 & &\text{on } \partial\Omega \times [0,T],\\
 u(x,0) = u_0(x), \;  \dot u(x,0) &= v_0(x), & &\text{in } \Omega,
\end{aligned}
\end{equation}
where $\Omega$ is a bounded Lipschitz domain in $\mathbb{R}^d$,  $\partial\Omega$ its boundary and $0 < c_a < a(x) < C_a$ a  piecewise constant function. If $\Omega_j$ and $\Omega_k$ are two subsets of $\Omega$ with the boundary $\Gamma_{jk}$ separating them and with $a \equiv a_k$ in $\Omega_k$ and $a \equiv a_j$ in $\Omega_j$, then if we denote by $u_j = u|_{\Omega_j}$ and $u_k = u|_{\Omega_k}$ we further have the transmission conditions
 \begin{equation}
   \label{eq:transmission}
u_j = u_k, \quad 
a_j\partial_{\mathbf{n}} u_j = a_k\partial_{\mathbf{n}} u_k, \qquad \text{ on } \Gamma_{jk},   
 \end{equation}
where $\mathbf{n}$ is the exterior normal to $\Omega_j$ (or $\Omega_k$).

We denote by $L^p(\Sigma)$, $1\le p\le +\infty$, the standard Lebesgue spaces, $\Sigma\subset\mathbb{R}^d$, $d \in \{ 2,3,4\}$, with corresponding norms $\|\cdot\|_{p,\Sigma}$; the norm of $L^2(\Sigma)$ will be denoted by $\ltwo{\cdot}{\Sigma}$. Further, $( \cdot,\cdot )_{\Sigma}$ denotes the standard $L^2$-inner product on $\Sigma$; when the arguments are vectors of $L^2$-functions, the $L^2$-inner product is modified in the standard fashion. 
We denote by $H^s(\Sigma)$ the standard Hilbertian Sobolev space of index $s\in\mathbb{R}$ of real-valued functions defined on $\Sigma\subset\mathbb{R}^d$; in particular $H^1_0(\Sigma)$ signifies the space of functions in $H^1(\Sigma)$ whose traces onto the boundary $\partial\Sigma$ vanish.  For $1\le p\le +\infty$, we denote the standard Bochner spaces by $L^p(0,T;X)$, with $X$ being a Banach space with norm $\|\cdot\|_X$.
%, consisting of all measurable functions $v: [0,T]\to X$, for which
%\beaa
%\|v\|_{L^p(0,T;X)}&:=&\Big(\int_0^T \|v(t)\|_{X}^p\ud t\Big)^{1/p}<+\infty,\quad \text{for}\quad 1\le p< +\infty, \\
%\|v\|_{L^{\infty}(0,T;X)}&:=&\ess\sup_{0\le t\le T}\|v(t)\|_{X}<+\infty,\quad \text{for}\quad  p = +\infty.
%\eeaa
Finally, we denote by $C(0,T;X)$ the space of continuous functions $v: [0,T]\to X$ with norm $\|v\|_{C(0,T;X)}:=\max_{0\le t\le T}\|v(t)\|_{X}<+\infty$.

Let $u_0 \in H^1_0(\Omega)$ and $v_0 \in L^2(\Omega)$, then \eqref{eq:wave} has a unique (weak) solution $u$  with 
\begin{equation}
  \label{eq:smoothness}
u \in L^2([0,T]; H^1_0(\Omega)), \quad  
\dot u \in L^2([0,T]; L^2(\Omega)), \quad  
\ddot u \in L^2([0,T]; H^{-1}(\Omega)), 
\end{equation}
see \cite[Theorem 8.1]{LioMI}.  Furthermore, see \cite[Theorem 8.2]{LioMI}, the solution is continuous in time with
\begin{equation}
  \label{eq:smoothness1}
u \in C([0,T]; H^1_0(\Omega)), \quad  
\dot u \in C([0,T]; L^2(\Omega)).
\end{equation}
\our{We denote the space of all solutions by
\begin{equation}
  \label{eq:soln_space}
  \solnspace = \left\{ u \;|\; u \text{ weak solution of \eqref{eq:wave}  with } u_0 \in H^1_0(\Omega), v_0 \in L^2(\Omega)\right\}.
\end{equation}}

\section{Space-time finite element space}\label{sec:space}
We aim to discretize this problem by a new space-time interior penalty discontinuous Galerkin method. In principle, this could be done on a general space-time mesh, however for the simplicity of presentation (and implementation) we construct a time discretization $0 = t_0 < t_1 < \dots<t_n\dots < t_N = T$ and locally quasi-uniform spatial-meshes $\mathcal{T}_n$ of $\Omega$ consisting of open $d$-dimensional simplices or $d$-box-type elements such that $\reply{\overline\Omega} = \cup_{K\in  \mathcal{T}_n} \overline K$, \reply{ with $K \cap \widetilde{K} = \emptyset$ for $K, \widetilde{K} \in \mathcal{T}_n$ and $K \neq \widetilde{K}$}. Therefore the space-time mesh consists of time-slabs  $\mathcal{T}_n \times I_n$, where $I_n = (t_n, t_{n+1})$ \our{and $\tau_n:=t_{n+1}-t_n$}.  %When clear from the context or of no influence, the subscript $n$ will be omitted. 

The discrete space-time mesh will consist of piecewise polynomials on each time-slab, given by the local space-time finite element space:
\[
\dspace{p} := \left\{u \in L^2(\Omega \times I_n)\;:\; u|_{K \times I_n} \in \mathcal{P}_{p}(\mathbb{R}^{d+1}),\; K \in \mathcal{T}_n\right\},
\]
where $\mathcal{P}_{p}$ is the space of polynomials of total degree $p$; the complete space-time finite element space on $\Omega \times [0,T]$, will be denoted by
\[
\dspaceT{p} := \{ u \in L^2(\Omega \times [0,T]) \;:\; u|_{\Omega \times I_n} \in \dspace{p}, \; n  =0,1,\dots, N-1\}.
\]

We require some notation. The {\em skeleton} of the mesh is defined by $\Gamma_n := \cup_{K \in \mathcal{T}_n} \partial K$ and the interior skeleton by $\Gamma_n^{\text{int}} = \Gamma_n \setminus \partial\Omega$. Moreover, we define the union of two skeletons of two subsequent meshes by $\hGamman:=\Gamma_{n-1} \cup\Gamma_n$; see Figure~\ref{figure:mesh}.

\begin{figure}[htbp]
 \centering
\scalebox{0.6}{
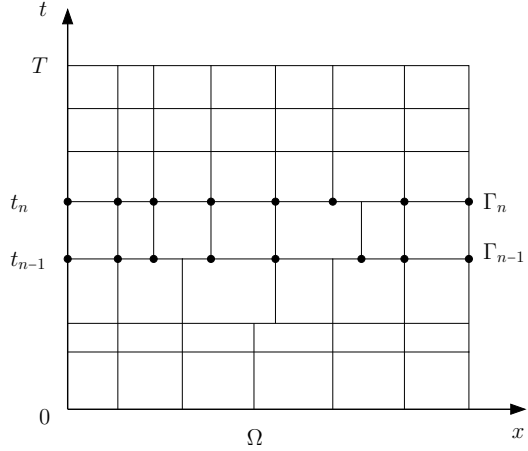
}
 \caption{\small An example of a space-time mesh. Skeletons $\Gamma_n$ and $\Gamma_{n-1}$ are highlighted by black dots. The union of the two we denote by $\hGamman$.}
 \label{figure:mesh}
\end{figure}

\begin{remark}
\reply{The numerical method presented below naturally admits variable spatial meshes from one space-time slab to the next. This flexibility will be crucial for future developments in the context of space-time adaptive dG methods, aiming to use local spatio-temporal resolution to resolve sharp moving fronts. If the spatial mesh remains fixed, all the formulas hold with $\Gamma = \Gamma_n = \hGamman$ for all $n$.}
\end{remark}

  Let $K^+$ and $K^-$ be two spatial elements sharing a face $e = \bar{K}^+\cap \bar{K}^- \reply{\subset} \Gamma_n^{\text{int}}$, with respective outward normal vectors $\nv^+$ and $\nv^-$ on $e$. For $u : \Omega \rightarrow \mathbb{R}$ and $\mathbf{v} : \Omega \rightarrow \mathbb{R}^d$,  let $u^{\pm} : e \rightarrow \mathbb{R}$ and $\mathbf{v}^\pm : e \rightarrow \mathbb{R}^d$ be the traces on $e$ with limits taken from $K^\pm$. We define the respective jumps and averages across each face $e\in \Gamma_n^{\text{int}}$ by
\[
\begin{aligned}
  \sa{u}|_e &= \tfrac12(u^++u^-), & \sa{\mathbf{v}}|_e &= \tfrac12(\mathbf{v}^++\mathbf{v}^-),\\
  \sj{u}|_e &= u^+\nv^++u^-\nv^-, & \sj{\mathbf{v}}|_e &= \mathbf{v}^+\cdot\nv^++\mathbf{v}^-\cdot\nv^-;
\end{aligned}
\]
if $e \reply{\subset  \overline{K^+} }\cap \partial\Omega$, we set $\sa{\mathbf{v}}|_e = \mathbf{v}^+$ and $\sj{u}|_e = u^+\nv^+$. Further, we define the temporal jump by
\[
\tj{u(t_n)} = u(t_n^+)-u(t_n^-), \qquad \tj{u(t_0)} = u(t_0^+).
\]
We will denote the spatial meshsize by $h : \Omega \times [0,T] \rightarrow \mathbb{R}$, defined by $h(x,t) = \operatorname{diam}(K)$ if $x \in K$ for $K \in {\mathcal T}_n$ and $t \in I_n$; when $x\in e = \bar{K}^+\cap \bar{K}^- $, we set $h(x,t):= \sa{h}$ to be the average. Finally we assume that there exists $c_{\mathcal T} > 0$ such that
\begin{equation}
  \label{eq:mesh_reg}
\diam(K)/\rho_K \leq c_{\mathcal T}, \qquad \forall K \in \mathcal{T}_n, \; n = 0,1,\dots, N-1,
\end{equation}
where $\rho_K$ is the radius of the inscribed circle of $K$.

\our{For simplicity of the presentation \emph{only} of the a priori error bounds below, we shall later make a shape-regularity assumption on the space-time  mesh  (cf., Assumption \ref{as:st_reg}).  We stress, however, that the stability results presented below do not depend on Assumption \ref{as:st_reg} and, therefore, the numerical method proposed below is unconditionally stable for any choice of spatial and temporal meshsizes. Indeed, an important advantage in using such space-time methods is that they do \emph{not} require any CFL-type restrictions.}

% The only results that depend on this assumption are the estimates in the standard norm in Proposition~\ref{prop:positiveE} and Corollary~\ref{cor:positiveE}.}

Finally, the broken spatial gradient will be denoted by $\nabla_n v$, given by $(\nabla_n v)|_K:= (\nabla v)|_K$ for all $K\in\mathcal{T}_n$ and a $v\in C(I_n;H^1_0(\Omega))+\dspace{p}$; collectively, we shall denote the broken gradient by $\babla v$  defined as
$(\babla v)|_{\Omega\times I_n}:= (\nabla_n v)|_{\Omega\times I_n}$, $n=0,\dots,N-1,$
for $v\in C(\prod_{n=0}^{N-1}\mathring{I}_n;H^1_0(\Omega))+\dspaceT{p}$, i.e., $v$ is allowed to be  discontinuous both in space and in time.

\section{A space-time discontinuous Galerkin method}\label{sec:method}
To derive the weak form suitable for dG discretisation we will follow an energy argument. We start by assuming that $u$ is a smooth enough solution of \eqref{eq:wave} and  let  $v \in \solnspace + \dspaceT{p}$. The standard symmetric interior penalty dG weak formulation on the time-slab $I_n$ when tested with $\dot v$ is given by
\begin{equation}\label{weak_IP_one}
\begin{split}
   (\ddot u,\dot v)_{\Omega \times I_n} &+ (a \babla u,\babla \dot v)_{\Omega \times I_n}
- (\sa{a \nabla u}, \sj{\dot v})_{\Gamma_n \times I_n}\\ 
&- (\sj{u},\sa{a \nabla \dot v})_{\Gamma_n \times I_n}+( \sigma_0 \sj{u} , \sj{\dot v} )_{\Gamma_n \times I_n} = 0,
\end{split}
\end{equation}
where 
\begin{equation}
  \label{eq:sigma_defn}
\sigma_0(x,t) = C_{\sigma_0} p^2(h(x,t))^{-1}  ,
\end{equation}
for a positive constant $C_{\sigma_0}$ to be made precise later.
This motivates the following definition of \emph{discrete energy} $E_h(t,v)$ at time $t \in I_n$, for \reply{$v\in \solnspace + \dspaceT{p}$}:
\[
E_h(t,v) := \tfrac12\|\dot v(t)\|^2_{\Omega}+\tfrac12\|\sqrt{a}\babla v(t)\|^2_{\Omega}
+\tfrac12\|\sqrt{\sigma_0}\sj{v(t)}\|^2_{\Gamma_n}-\lip{\sa{a \nabla v(t)}}{ \sj{ v(t)}}{\Gamma_n}.
\]
Using the classical inverse inequality
$
  \|v\|^2_{\partial K} \leq C_{\text{inv}} p^2 |\partial K|/|K| \|v\|^2_K$, for all $ v \in \mathcal{P}^p(K)$,  (see, e.g., \cite{PieE,SchwabBook},) we see that
\[
\begin{split}  
2\int_{\Gamma_n} |\sa{a \nabla v(t,\reply{s})}|^2 ds &\leq  C_a^2 \sum_{K \in \mathcal{T}_n} \int_{\partial K} |\nabla v(t,\reply{s})|^2 ds\\
&\leq  \sum_{K \in \mathcal{T}_n} \frac{C_a^2C_{\text{inv}}p^2|\partial K|}{|K|}\int_{ K}  |\nabla v(t,\reply{x})|^2 dx\\
&\leq  \sum_{K \in \mathcal{T}_n} \frac{c_{\mathcal T} C_a^2 C_{\text{inv}}p^2}{c_a h_K}\int_{ K} |\sqrt{a}  \nabla v(t,\reply{x})|^2 dx.
\end{split}
\]
Hence,  if the penalisation parameter $C_{\sigma_0}$ is chosen large enough: in particular 
\begin{equation}
  \label{eq:Cstab_ineq}
%C_{\sigma_0} \geq c_\mathcal{T} C_a C_{\text{inv}}/(\reply{4}c_a),  
C_{\sigma_0} \geq c_\mathcal{T} C_a^2 C_{\text{inv}}/c_a,  
\end{equation}
suffices we have that
\begin{equation}\label{coercivity_inverse}
  \begin{split}    
|\lip{\sa{a \nabla v(t)}}{ \sj{ v(t)}}{\Gamma_n}| &\le \tfrac12  \|\sqrt{\sigma_0}\sj{v(t)}\|^2_{\Gamma_n}+\tfrac12 \int_{\Gamma_n}\sigma_0^{-1} |\sa{a \nabla v(t,\reply{s})}|^2 ds\\
\\&\le \tfrac12\|\sqrt{\sigma_0}\sj{v(t)}\|^2_{\Gamma_n} +\tfrac{1}{\reply{4}}\|\sqrt{a}\babla v(t)\|^2_{\Omega},
  \end{split}
\end{equation}
ensuring the \reply{non-negativity} of the energy $E_h(t,v)$ for functions in \reply{$\dspaceT{p}$}:
\[
E_h(t,v) \geq   \tfrac12\|\dot v(t)\|^2_{\Omega}+\tfrac14\|\sqrt{a}\babla v(t)\|^2_{\Omega}, \qquad \text{for all } v \in \mathcal{X}+\dspaceT{p}.
\]
% \begin{remark}
%   If $a(x)$ is constant within an element $K$, as we will assume in the next section, the occurances of $C_a$ and $c_a$ in \eqref{eq:Cstab_ineq} can be removed.
%ONLY if $a$ is not a matrix!!
% \end{remark}

 Choosing as test function $v = u$ in \eqref{weak_IP_one} and summing over $n$, we obtain
\[
\begin{split}
 0 &= \sum_{n = 0}^{N-1} \int_{I_n}\frac{d}{dt} \Big(\tfrac12 \ltwo{\dot u}{\Omega}^2+\tfrac12 \ltwo{\sqrt{a} \babla u}{\Omega}^2- \lip{\sa{a\nabla u}}{\sj{ u}}{\Gamma_n} + \tfrac12 \ltwo{\sqrt{\sigma_0} \sj{u}}{\Gamma_n}^2 \Big) dt
\\
&= E_h(t_N^-,u)-E_h(t_0^+,u)-\sum_{n = 1}^{N-1}\tj{E_h(t_n,u)}.
\end{split}
\] 
In order to allow for discontinuities in time, the formulation \eqref{weak_IP_one} needs to be modified (in a consistent fashion) to control the terms $\tj{E_h(t_n,u)}$ that have no sign. To this end, we shall use the elementary algebraic identity
\begin{equation}
  \label{eq:upwind}
  \begin{split}    
 -\tj{ f(u(t_n))g(u(t_n))} +\tj{ f(u(t_n))}g(u(t_n^+))
&+ \tj{ g(u(t_n))}f(u(t_n^+)) \\
  &= \tj{f(u(t_n))}\tj{g(u(t_n))},
  \end{split}
\end{equation}
\reply{for some scalar quantities $f,g$ which may, in general, be discontinuous across $t_n$. The idea here is to add terms that change the jump of a product to the product of jumps in the above energy identity, without compromising its consistency. For instance, to $\lip{\ddot u}{ \dot v}{\Omega \times I_n} $  in \eqref{weak_IP_one} we add the extra term $\lip{\tj{\dot u(t_n)}}{\dot v(t_n^+)}{\Omega}$.
For $n > 0$, this does not change the consistency of \eqref{weak_IP_one} with respect to \eqref{eq:wave} in weak form, as the smoothness assumptions on the initial data ensure that the exact solution satisfies $\dot u \in C([0,T]; L^2(\Omega))$;  for $n = 0$, we have by the initial condition $\dot u(t_0) = v_0$ that
\[
\lip{\tj{\dot u(t_0)}}{\dot v(t_0^+)}{\Omega}
= \lip{\dot u(t_0^+)}{\dot v(t_0^+)}{\Omega}
= \lip{v_0}{\dot v(t_0^+)}{\Omega}.
\]
Hence, for consistency, we add the term $\lip{\dot u_0}{\dot v(t_0^+)}{\Omega}$ to the right-hand side also.
Using \eqref{eq:upwind} now, for $u=v=w$, with $w$ piecewise sufficiently smooth function, we have
\[
\begin{split}  
\sum_{n = 0}^{N-1}\lip{\ddot w}{ \dot w}{\Omega \times I_n}  +\lip{\tj{\dot w(t_n)}}{\dot w(t_n^+)}{\Omega}
=& \tfrac12\ltwo{\dot w(t^-_N)}{\Omega}^2+\sum_{n = 0}^{N-1}\lip{\tj{\dot w(t_n)}}{\dot w(t_n^+)}{\Omega} -\tj{\tfrac12\ltwo{\dot w(t_n)}{\Omega}^2} \\
=&\tfrac12\ltwo{\dot w(t_N^-)}{\Omega}^2+\tfrac12\ltwo{\dot w(t_0^+)}{\Omega}^2+
\tfrac12\sum_{n = 1}^{N-1}\ltwo{\tj{\dot w(t_n)}}{\Omega}^2,
%\\
%=&\tfrac12\ltwo{\dot w(t_N^-)}{\Omega}^2-\tfrac12\ltwo{\dot w(t_0^+)}{\Omega}^2+\ltwo{v_0}{\Omega}^2-\tfrac12\sum_{n = 1}^{N-1}\ltwo{\tj{\dot w(t_n)}}{\Omega}^2,
\end{split}
\]
with the additional terms contributing to energy dissipation leading to a stable method.
Completely analogous considerations lead to addition of corresponding terms to treat the second and the last terms on the left-hand side of \eqref{weak_IP_one}. For the remaining third and fourth terms on the left-hand side of \eqref{weak_IP_one}, we include the additional terms $ - \lip{\tj{\sa{a\babla u(t_n)}}}{\sj{ v(t_n^+)}}{\hGamman}$ and $- \lip{\tj{\sj{u(t_n)}}}{\sa{a\nabla  v(t_n^+)}}{\hGamman}$; note that these again do not change the consistency due to both $u\in C([0,T]; H^1_0(\Omega))$ and \eqref{eq:transmission}. The use of $\hGamman = \Gamma_{n-1} \cup \Gamma_n$ in the last terms also merits a brief explanation: since we assume that both the solution and its spatial flux are continuous within a space-time element $K \times I_n$, $K \in \mathcal{T}_n$ it follows that
\[
 \lip{\sa{a \nabla u}}{\sj{\dot v}}{\Gamma_n \times I_n} =  \lip{\sa{a \nabla u}}{\sj{\dot v}}{\hGamman \times I_n}\ \text{ and }\ 
 \lip{\sj{u}}{\sa{a \nabla \dot v}}{\Gamma_n \times I_n}
= \lip{\sj{u}}{\sa{a \nabla \dot v}}{\hGamman \times I_n}.
\]
}
\reply{For consistency, the terms $- \lip{ \sa{a\nabla u_0}}{\sj{ v(t_0^+)}}{\Gamma_0}- \lip{ \sj{u_0}}{\sa{a\nabla  v(t_0^+)}}{\Gamma_0}$ are also added to the right-hand side of \eqref{weak_IP_one}. The above identity can again be used to show that we will have terms of the type $\lip{\tj{\sa{a\babla u(t_n)}}}{\tj{\sj{ u(t_n)}}}{\hGamman}$ in the energy identity; these do not have a sign but can be bounded by the other terms in the energy using \eqref{coercivity_inverse}.}

In view of the above considerations, we can now state the space-time weak formulation of our method:
\begin{equation}
  \label{eq:formulation}
  \begin{aligned}
\sum_{n = 0}^{N-1}\lip{\ddot u}{ \dot v}{\Omega \times I_n} &+\lip{\tj{\dot u(t_n)}}{\dot v(t_n^+)}{\Omega} 
+\lip{a \babla u}{ \babla \dot v}{\Omega \times I_n} + \lip{\tj{a \babla u(t_n)}}{\babla v(t_n^+)}{\Omega}\\
&- \lip{\sa{a \nabla u}}{\sj{\dot v}}{\Gamma_n \times I_n} - \lip{\tj{\sa{a\babla u(t_n)}}}{\sj{ v(t_n^+)}}{\hGamman}  
\\ &- \lip{\sj{u}}{\sa{a \nabla \dot v}}{\Gamma_n \times I_n}  
- \lip{\tj{\sj{u(t_n)}}}{\sa{a\nabla  v(t_n^+)}}{\hGamman} 
\\ &+ \lip{\sigma_0  \sj{u}}{\sj{\dot v}}{\Gamma_n \times I_n} + \lip{\sigma_0 \tj{ \sj{u(t_n)}}}{ \sj{ v(t_n^+)}}{\hGamman} 
\\ &+  \lip{\sigma_1 \sj{u}}{\sj{v}}{\Gamma_n \times I_n}
+\lip{\sigma_2 \sj{a\nabla u}}{\sj{a\nabla v}}{\Gamma_n \times I_n}  = \binit(v),
  \end{aligned}
\end{equation}
\reply{where  $\binit$ is given by
\begin{equation}
  \label{eq:rhsform_init}
  \begin{aligned}    
  \binit(v) := & 
\lip{ v_0}{ \dot v(t_0^+)}{\Omega} + \lip{a \babla u_0}{ \babla v(t_0^+)}{\Omega}
- \lip{ \sa{a\nabla u_0}}{\sj{ v(t_0^+)}}{\Gamma_0}\\
&- \lip{ \sj{u_0}}{\sa{a\nabla  v(t_0^+)}}{\Gamma_0}
+ \lip{\sigma_0   \sj{u_0}}{\sj{ v(t_0^+)}}{\Gamma_0}.
  \end{aligned}  
\end{equation}}
Note that the last two terms in the definition of $\binit$ are zero if the initial data is continuous in space.
The terms with positive penalty parameters $\sigma_1$ and $\sigma_2$ in the weak formulation \eqref{eq:formulation} do not affect the consistency of the weak formulation; the need for their inclusion in the method will become apparent in the convergence analysis.

\reply{ Although the choice of $\sigma_0$, $\sigma_1$, and $\sigma_2$ will only become evident in the analysis, for easy referral we state the choice of the stabilisation parameters $\sigma_0$, $\sigma_1$, and $\sigma_2$ here also: $\sigma_0$ is chosen as \eqref{eq:sigma_defn}, (see also \eqref{eq:Cstab_ineq},) and we set
\[
\sigma_1|_{\partial K\cap \Gamma_n\times I_n}= C_a\frac{ p^3}{h\tau_n} , \qquad
\sigma_2 = \frac{h}{C_a\tau_n}.
\]
}

Thus, we arrive at a space-time discrete method, which can be thought of in two ways: as a method for obtaining a discrete solution on a fixed space-time domain $\Omega \times [0,T]$, or as a time-stepping method. The former viewpoint will be  useful in obtaining convergence estimates, while the latter in implementing the method. Consequently, we define the following three bilinear forms to describe these two viewpoints:
\begin{equation}
  \label{eq:lhsform}
  \begin{aligned}    
  a_n(u,v) &:=\
   \lip{\ddot u}{\dot v}{\Omega \times I_n}+ \lip{\dot u(t^+_n)}{ \dot v(t_n^+)}{\Omega} \\
   &+\lip{a\babla u}{\babla \dot v}{\Omega \times I_n}+ \lip{ a\babla u(t_n^+)}{ \babla v(t_n^+)}{\Omega}\\
   &- \lip{\sa{a\nabla u}}{\sj{\dot v}}{\Gamma_n \times I_n}- \lip{\sa{a\nabla u(t_n^+)}}{\sj{ v(t_n^+)}}{\Gamma_n}  \\
   &- \lip{\sj{u}}{ \sa{a\nabla \dot v}}{\Gamma_n \times I_n} - \lip{\sj{u(t_n^+)}}{\sa{a\nabla  v(t_n^+)}}{\Gamma_n}  \\
   &+ \lip{\sigma_0   \sj{u}}{ \sj{\dot v}}{\Gamma_n \times I_n} + \lip{\sigma_0  \sj{u(t_n^+)}}{\sj{ v(t_n^+)}}{\Gamma_n} \\
&+ \lip{\sigma_1   \sj{u}}{ \sj{ v}}{\Gamma_n \times I_n}
+ \lip{\sigma_2   \sj{a \nabla u}}{ \sj{a \nabla v}}{\Gamma_n \times I_n},
  \end{aligned}
\end{equation}
\begin{equation}
  \label{eq:rhsform}
  \begin{aligned}    
  b_n(u,v) :=&\  
\lip{ \dot u(t^-_n)}{ \dot v(t_n^+)}{\Omega} + \lip{a\babla u(t_n^-)}{ \babla v(t_n^+)}{\Omega}
- \lip{ \sa{a\nabla u(t_n^-)}}{\sj{ v(t_n^+)}}{\Gamma_n}\\
&- \lip{ \sj{u(t_n^-)}}{\sa{a\nabla  v(t_n^+)}}{\Gamma_{n-1}}
+ \lip{\sigma_0   \sj{u(t_n^-)}}{\sj{ v(t_n^+)}}{\hGamman},
  \end{aligned}  
\end{equation}
 and
\begin{equation}
  \label{eq:aform}  
a(u,v) := \sum_{n = 0}^{N-1} a_n(u,v)-\sum_{n = 1}^{N-1}b_n(u,v),
\end{equation}
which just gives the left-hand side in \eqref{eq:formulation}. 
\begin{definition}\label{defn:time_stepping}
Given subspaces $X_n \subseteq \dspace{p}$, the time-stepping method is described by: find $u^n \in X_n$, $n = 1, 2, \dots, N-1$,  such that
\begin{equation}
  \label{eq:timestep}
a_n(u^n,v) = b_n(u^{n-1},v), \qquad  \text{ for all }v \in X_n
\end{equation}
\reply{and
\begin{equation}
  \label{eq:timestep_initial}
  a_0(u^0,v) = \binit(v), \qquad  \text{ for all }v \in X_0.
\end{equation}}
 Equivalently, given a subspace $X \subseteq \dspaceT{p}$, the full space-time discrete system can be presented as: find $u \in X$  such that
\begin{equation}
  \label{eq:fullsys}
  a(u,v) = \binit(v), \qquad \text{for all } v \in X.
\end{equation}  
\end{definition}
\begin{lemma}
  The following identities hold for \reply{any $w \in \solnspace+ \dspaceT{p}$}:
\begin{equation}
  \label{eq:a_nuu}
a_n(w,w) = E_h(t_{n+1}^-,w)+E_h(t_n^+,w)+\ltwo{\sqrt{\sigma_1}   \sj{w}}{\Gamma_n \times I_n}^2
+ \ltwo{\sqrt{\sigma_2}   \sj{a \nabla w}}{\Gamma_n \times I_n}^2,    
\end{equation}
for $n = 0,1,\dots,N-1$, and
\begin{equation}
  \label{eq:dg_norm_expr}  
\begin{split}  
a(w,w) = E_h(t_N^-,w)&+E_h(t_0^+,w) +\sum_{n = 1}^{N-1} \Big(\tfrac12 \ltwo{ \tj{\dot w(t_n)}}{\Omega}^2+\tfrac12\ltwo{\sqrt{a}\tj{ \babla w(t_n)}}{\Omega}^2 \\
&-\lip{\tj{\sa{a \babla w(t_n)}}}{\tj{\sj{w(t_n)}}}{\hGamman}+\tfrac12\ltwo{ \tj{\sqrt{\sigma_0}\sj{w(t_n)}}}{\hGamman}^2\Big)\\
&+ \sum_{n = 0}^{N-1}\Big(\ltwo{\sqrt{\sigma_1} \sj{w}}{\Gamma_n \times I_n}^2
+ \ltwo{\sqrt{\sigma_2} \sj{a \nabla w}}{\Gamma_n \times I_n}^2\Big).
\end{split}
\end{equation} 
\end{lemma}
\begin{proof}
  The identities follow from the definitions of the bilinear forms and the energy $E_h(t,w)$.
\end{proof}

Next we investigate the consistency and stability of the discrete scheme.

\begin{theorem}[Consistency and Stability]\label{thm:const_stab} Let the space $\dspaceT{p}$ be given. Then, the following statements hold:
  \begin{enumerate}
  \item  \label{thm:consistency} Let $u$ be the weak solution of \eqref{eq:wave}, with $u_0 \in H^1_0(\Omega)$ and $v_0 \in L^2(\Omega)$. Then $u$ satisfies \eqref{eq:fullsys}.
\item  \label{thm:stability} \reply{For $C_{\sigma_0}$  satisfying} \eqref{eq:Cstab_ineq} and for any $v \in \dspaceT{p}$ and $t \in (0,T)$, the energy $E_h(t,v)$ 
is bounded from below by
\begin{equation}\label{energy_positive}
E_h(t,v) \geq  \tfrac12\ltwo{\dot v(t)}{\Omega}^2+\tfrac14\ltwo{\sqrt{a}\babla v(t)}{\Omega}^2. 
%+\frac14\ltwo{\sqrt{\sigma_0}\sj{v(t)}}{\Gamma_n}^2.
\end{equation}
%Further, let $u^n \in \dspace{p}$, $n = 0,\dots, N-1$,  satisfy \eqref{eq:formulation}. Then,
Further, if $X$ is a subspace of $\dspaceT{p}$ and $U \in X$ is the discrete solution, i.e., satisfies \eqref{eq:fullsys}, then 
$
E_h(t_N^-,U) \leq E_h(t_1^-,U).
$
  \end{enumerate}
\end{theorem}
\begin{proof}
Statement \ref{thm:consistency} follows from the derivation of the formulation and the regularity of the unique solution $u$; see \eqref{eq:smoothness} and \eqref{eq:smoothness1}.  We have already shown the positivity of the energy under the condition on $C_{\sigma_0}$; see \eqref{coercivity_inverse}. 

To prove the remaining statement we proceed as follows. \reply{Combining \eqref{eq:fullsys} with \eqref{eq:dg_norm_expr} gives} the energy identity 
\begin{equation}
  \label{eq:dis_energy}
  \begin{split}    
E_h(t_N^-,U) = \ &  \binit(U)-E_h(t_0^+,U)-\sum_{n = 1}^{N-1} \tfrac12 \ltwo{\tj{\dot U(t_n)}}{\Omega}^2+\tfrac12\ltwo{\sqrt{a}\tj{\babla U(t_n)}}{\Omega}^2 \\
&+\sum_{n = 1}^{N-1} \lip{\tj{\sa{a \babla U(t_n)}}}{\tj{\sj{U(t_n)} }}{\hGamman}-\tfrac12\ltwo{\sqrt{\sigma_0} \tj{\sj{U(t_n)}}}{\hGamman}^2\\
&- \sum_{n = 0}^{N-1}\ltwo{\sqrt{\sigma_1}\sj{U}}{\Gamma_n \times I_n}^2
- \sum_{n = 0}^{N-1}\ltwo{\sqrt{\sigma_2}\sj{a \nabla U}}{\Gamma_n \times I_n}^2.
  \end{split}
\end{equation}

Expression \eqref{eq:a_nuu}
implies that
\[
a_0(U,U) = E_h(t_{1}^-,U)+E_h(t_0^+,U)+\ltwo{\sqrt{\sigma_1}   \sj{U}}{\Gamma_0 \times I_0}^2
+ \ltwo{\sqrt{\sigma}_2   \sj{a \nabla U}}{\Gamma_0 \times I_0}^2 = \binit(U).
\]
Hence, the energy identity \eqref{eq:dis_energy} can be written as
\begin{equation}
  \label{eq:dis_energy1}
  \begin{split}    
E_h(t_N^-,U) =&\  E_h(t_1^-,U)-\sum_{n = 1}^{N-1} \Big(\tfrac12 \ltwo{\tj{\dot U(t_n)}}{\Omega}^2+\tfrac12\ltwo{\sqrt{a}\tj{\babla U(t_n)}}{\Omega}^2 \\
&-\lip{\tj{\sa{a\babla U(t_n)}}}{\tj{\sj{U(t_n)} }}{\hGamman}+\tfrac12\ltwo{\sqrt{\sigma_0} \tj{\sj{U(t_n)}}}{\hGamman}^2\\
&+ \ltwo{\sqrt{\sigma_1}\sj{U}}{\Gamma_n \times I_n}^2
+ \ltwo{\sqrt{\sigma_2}\sj{a \nabla U}}{\Gamma_n \times I_n}^2\Big).
  \end{split}
\end{equation}
 Arguments used to prove the non-negativity of the discrete energy \eqref{coercivity_inverse}, also show that the above equality implies that the discrete energy decreases at each time-step.
\end{proof}

\section{Polynomial Trefftz spaces}\label{sec:Tspace}

We shall consider the discrete space of local polynomial solutions to the wave equation, where we make an additional assumption on the mesh and on $a(x)$ that allows us to define the Trefftz spaces. \reply{Such polynomial spaces have already appeared in literature; see for example \cite{KreMPS,MacJ,Whi03}.}

\begin{assumption}
Let the diffusion coefficient $a(\cdot)$ and the mesh be such that $a(\cdot)$ is constant  in each element $K \in \mathcal{T}_n$ for each $n$.  
\end{assumption}

\begin{definition}[Polynomial Trefftz spaces]\label{defn:trefftz_space}
Let $\Tspace \subseteq \dspace{p}$ be a subspace of functions satisfying the homogeneous wave equation on any space-time element $K \times I_n$:
\[
\Tspace: = \left \{ v \in \dspace{p} \;:\; \ddot v(t,x) -\nabla \cdot \left( a \,\nabla v \right)(t,x) = 0, \quad t \in I_n, x \in K,\,  K\in \mathcal{T}_n\right\}.
\]  
The space on $\Omega \times [0,T]$ is then defined as
\[
\TspaceT = \left\{ u \in L^2(\Omega \times [0,T]) \;:\; u|_{\Omega \times I_n} \in \Tspace, \; n = 0,1\dots, N-1\right\} \subseteq \dspaceT{p}.
\]
\end{definition}
Polynomial plane waves are examples of functions in this space
\[
(t+ a^{-1/2}\alpha \cdot  x )^j, \qquad |\alpha| = 1, \alpha \in \mathbb{R}^d, j \in \{0,\dots,p\}.
\]

\begin{proposition}\label{prop:trefftz_poly}
  The local dimension of the Trefftz space in $\mathbb{R}^d$ is given by
\[
\operatorname{dim}(\Tspace(K)) = \left\{
  \begin{array}{cc}
    2p+1 & d = 1\\
    (p+1)^2 & d = 2\\
    \tfrac16 (p+1)(p+2)(2p+3) & d = 3
  \end{array}
\right..
\]
\end{proposition}
\begin{proof}
  The proof for $d = 1$ is clear. For $d = 3$ in \cite{Whi03} it is shown that the dimension of \reply{the space of  Trefftz, homogeneous polynomials of degree}  $j$ is $(j+1)^2$, hence the total dimension is given by
\[
\sum_{j = 0}^p (j+1)^2 = \tfrac16 (p+1)(p+2)(2p+3).
\]
The case $d = 2$ is proved similarly by noticing that the dimension of \reply{the space of Trefftz, homogeneous polynomials} of  degree $j$ is $2j+1$.
\end{proof}

\begin{example}
In two dimensions, \reply{with a constant, scalar diffusion coefficient $a > 0$},  the following is a  basis for $\Tspace$ with $p = 3$
\[
\begin{split}  
\{ &1,t,x,y,tx,ty,xy,at^2+x^2,at^2+y^2, \\ &xyt, at^3+3x^2t, x^3+3at^2x, y^3+3at^2y,(at^2+x^2)y, (at^2+y^2)x,(x^2-y^2)t\}.
\end{split}
\]
\end{example}
\subsection{Existence and uniqueness}
Next we prove that the discontinuous Galerkin energy norm is indeed a norm on the subspace of Trefftz polynomials. This also includes piecewise linear polynomials as $\TspaceT = \dspaceT{p}$ for $p = 1$.

\begin{proposition}\label{prop:seminorms}
With the choice of $C_{\sigma_0}$ as in \eqref{eq:Cstab_ineq} and $\sigma_1, \sigma_2 > 0$,
bilinear forms $a_n(\cdot,\cdot)$ and $a(\cdot,\cdot)$ give rise to two semi-norms
\[
\dgnorm{v}_n := \big(a_n(v,v)\big)^{1/2}, \qquad v \in \dspace{p}
\] 
and
\[
\dgnorm{v} := \big(a(v,v)\big)^{1/2}, \qquad v \in \dspaceT{p}.
\] 
These are in fact norms on Trefftz subspaces $\Tspace$ and $\TspaceT$.
\end{proposition}
\begin{proof}
Recalling \eqref{eq:a_nuu}  and using  \eqref{energy_positive}, we deduce that 
$\dgnorm{v}_n^2 \geq 0$ and is hence a semi-norm.

Suppose $\dgnorm{v}_n = 0$ for $v \in \Tspace$. Then, $a \nabla v$  and $v$ have no jumps across the space skeleton and hence $v$ is a weak solution of the homogeneous wave equation on $\Omega \times I_n$ with zero initial and boundary conditions. Uniqueness implies $v \equiv 0$ and hence that $\dgnorm{\cdot}_n$ is a norm on this space.

The analysis of $\dgnorm{\cdot}$ is similar recalling \eqref{eq:dg_norm_expr}, which shows that $\dgnorm{\cdot}$ is a semi-norm if the stabilization parameter is chosen correctly. Proceeding as in the first case, shows that it is in fact a norm on the Trefftz spaces.
\end{proof}

\begin{corollary}\label{cor:ex_un}
Under the conditions of the above proposition and with initial data $u_0 \in H^1_0(\Omega)$, $v_0 \in L^2(\Omega)$, the discrete system \eqref{eq:fullsys} with $X = \TspaceT$ has a unique solution.
\end{corollary}
\begin{proof}
  The uniqeness of the solution to \eqref{eq:fullsys} over the Trefftz space $X = \TspaceT$ follows from $a(\cdot,\cdot)$ being a norm on this space. Existence of the solution to the linear system follows from uniqueness.
\end{proof}

Next, we present the convergence analysis of the Trefftz based method.

\subsection{Convergence analysis}
We shall now establish the quasi-optimality of the proposed method.

\begin{proposition}
Let \our{$w \in \solnspace + \TspaceT$} and $v \in \TspaceT$, then 
\[
|a(w,v)| \leq C_\star \dgnorml{w}\dgnorm{v},
\]  
for some constant $C_\star > 0$ and 
\[
\begin{split}
\dgnorml{w}^2 =& \tfrac12\sum_{n = 1}^{N} \Big(\|\dot w(t_n^-)\|^2_{\Omega}+\|\sqrt{a}\nabla w(t_n^-)\|^2_{\Omega}
+\|\sqrt{\sigma_0}\sj{w(t_n^-)}\|^2_{\Gamma_n}+\|\sigma_0^{-1/2}\sa{a\nabla w(t_n^-)}\|^2_{\Gamma_n}\Big)\\
&+\sum_{n = 0}^{N-1}\left(\ltwo{\sqrt{\sigma_1}\sj{w}}{\Gamma_n \times I_n}^2
+\ltwo{\sqrt{\sigma_2}\sj{a\nabla w}}{\Gamma_n \times I_n}^2
+\ltwo{\sigma_2^{-1/2}\sa{\dot w}}{\Gamma_n^{\rm int} \times I_n}^2\right.\\
&\qquad\qquad+\left.\ltwo{\sigma_1^{-1/2} \sa{a\nabla \dot w}}{\Gamma_n \times I_n}^2+\ltwo{\sigma_0\sigma_1^{-1/2}\sj{\dot w}}{\Gamma_n \times I_n}^2\right).
\end{split}
\]
\end{proposition}
\begin{proof}
Integration by parts gives
\[
\begin{split}
 \lip{\ddot w}{ \dot v}{\Omega\times I_n} + \lip{a\nabla w}{ \nabla \dot v}{\Omega\times I_n} 
 =\  &-\lip{\dot w}{ \ddot v}{\Omega\times I_n} - \lip{a\nabla \dot w}{ \nabla  v}{\Omega\times I_n} \\
&+\lip{\dot w(t_{n+1}^-)}{ \dot v(t_{n+1}^-)}{\Omega}-\lip{\dot w(t_{n}^+)}{ \dot v(t_{n}^+)}{\Omega}\\
& + \lip{a\nabla w(t_{n+1}^-)}{ \nabla v(t_{n+1}^-)}{\Omega}-\lip{a\nabla w(t_n^+)}{ \nabla v(t_n^+)}{\Omega}\\
=\ &-\lip{\sj{\dot w}}{\sa{a \nabla v}}{\Gamma_n\times I_n}-\lip{\sa{\dot w}}{\sj{a \nabla v}}{\Gamma_n^{\text{int}}\times I_n}\\
&+\lip{\dot w(t_{n+1}^-)}{ \dot v(t_{n+1}^-)}{\Omega}-\lip{\dot w(t_{n}^+)}{ \dot v(t_{n}^+)}{\Omega}\\
& + \lip{a\nabla w(t_{n+1}^-)}{ \nabla v(t_{n+1}^-)}{\Omega}-\lip{a\nabla w(t_n^+)}{ \nabla v(t_n^+)}{\Omega},
\end{split}
\]
since $v\in \TspaceT$ and by using the (elementary) identity
\[
-\lip{a\nabla \dot w}{ \nabla  v}{\Omega\times I_n}  = \lip{\dot w}{ \nabla\cdot a\nabla  v}{\Omega\times I_n} 
-\lip{\sj{\dot w}}{\sa{a \nabla v}}{\Gamma_n\times I_n}-\lip{\sa{\dot w}}{\sj{a \nabla v}}{\Gamma_n^{\text{int}}\times I_n},
\]
in the second step. Further applications of integration by parts in time yield
\[
\begin{split}
&-\lip{\sj{\dot w}}{\sa{a \nabla v}}{\Gamma_n\times I_n} \\
=\ & \lip{\sj{ w}}{\sa{a \nabla \dot v}}{\Gamma_n\times I_n}-\lip{\sj{ w(t_{n+1}^-)}}{\sa{a \nabla  v(t_{n+1}^-)}}{\Gamma_n} +\lip{\sj{ w(t_{n}^+)}}{\sa{a \nabla  v(t_{n}^+)}}{\Gamma_n},
\end{split}
\]
and
\[
\begin{split}  
&-\lip{ \sa{a\nabla w} }{ \sj{\dot v}}{\Gamma_n\times I_n} +\lip{\sigma_0 \sj{w}}{ \sj{\dot v}}{\Gamma_n\times I_n}\\
= \ & \lip{\sa{a\nabla \dot w}}{ \sj{v}}{\Gamma_n\times I_n} -\lip{\sigma_0 \sj{\dot w}}{\sj{ v}}{\Gamma_n\times I_n} \\
&-\lip{ \sa{a\nabla w(t_{n+1}^-)}}{ \sj{ v(t_{n+1}^-)}}{\Gamma_n}
+\lip{ \sa{a\nabla w(t_{n}^+)}}{ \sj{ v(t_{n}^+)}}{\Gamma_n}\\
&+\lip{\sigma_0\sj{w(t_{n+1}^-)}}{ \sj{ v(t_{n+1}^-)}}{\Gamma_n}
-\lip{\sigma_0 \sj{w(t_{n}^+)}}{ \sj{ v(t_{n}^+)}}{\Gamma_n}.
\end{split}
\]
Substituting these into% the definition of $a_n(\cdot,\cdot)$ 
\eqref{eq:lhsform}, we obtain
\begin{equation}
  \label{eq:trefftz_bfn}  
  \begin{aligned}    
  a_n(w,v) =& \ 
    \lip{\sa{a\nabla \dot w}}{\sj{ v}}{\Gamma_n \times I_n}- \lip{\sa{a\nabla w(t_{n+1}^-)}}{\sj{ v(t_{n+1}^-)}}{\Gamma_n} \\
&- \lip{\sigma_0   \sj{\dot w}}{\sj{v}}{\Gamma_n \times I_n}+ \lip{\sigma_0  \sj{w(t_{n+1}^-)}}{\sj{ v(t_{n+1}^-)}}{\Gamma_n}\\
&- \lip{ \sa{\dot w}}{ \sj{a \nabla v}}{\Gamma_n^{\rm int}\times I_n}- \lip{\sj{w(t_{n+1}^-)}}{\sa{a\nabla  v(t_{n+1}^-)}}{\Gamma_n} \\
&+ \lip{\dot w(t^-_{n+1})}{\dot v(t_{n+1}^-)}{\Omega} + \lip{a\nabla w(t_{n+1}^-)}{ \nabla v(t_{n+1}^-)}{\Omega}\\
&
+ \lip{\sigma_1  \sj{w}}{\sj{ v}}{\Gamma_n \times I_n}
+ \lip{\sigma_2  \sj{a\nabla w}}{\sj{a \nabla v}}{\Gamma_n \times I_n}.
  \end{aligned}
\end{equation}
Therefore, upon adopting the notational convention 
$
\tj{f(t_N^-)} := f(t_N^-), 
$
we have
\begin{equation}
  \label{eq:trefftz_bf}  
\begin{split}
  a(w,v) =& \sum_{n = 0}^{N-1}a_n(w,v) - \sum_{n = 1}^{N-1}b_n(w,v)\\
=&  \sum_{n = 0}^{N-1}\Big( \lip{\sa{a\nabla\dot  w}}{ \sj{ v}}{\Gamma_n\times I_n}
- \lip{\sigma_0   \sj{\dot w}}{ \sj{v}}{\Gamma_n\times I_n} 
- \lip{ \sa{\dot w}}{ \sj{a\nabla v}}{\Gamma_n^{\rm int}\times I_n}  \\
&+ \lip{\sigma_1  \sj{w}}{ \sj{ v}}{\Gamma_n \times I_n}
+\lip{\sigma_2  \sj{a\nabla w}}{ \sj{a \nabla v}}{\Gamma_n \times I_n}\Big)\\
%&+ \lip{\dot w(t^-_N)}{\dot v(t_N^-)}{\Omega} + \lip{a\nabla_{N-1} w(t_N^-)}{ \nabla_{N-1} v(t_N^-)}{\Omega}\\
%&- \lip{\sa{a\nabla w(t_N^-)}}{\sj{ v(t_N^-)}}{\Gamma_{N-1}} 
%- \lip{ \sj{w(t_N^-)}}{ \sa{a\nabla  v(t_N^-)}}{\Gamma_{N-1}}\\
%&+ \lip{\sigma_0  \sj{w(t_N^-)}}{\sj{ v(t_N^-)}}{\Gamma_{N-1}}\\
&- \sum_{n = 1}^{N}\Big(\lip{ \dot w(t_n^-)}{ \tj{\dot v(t_n)}}{\Omega} + \lip{a\nabla w(t_n^-)}{ \tj{\nabla v(t_n)}}{\Omega}\\
&- \lip{\sa{a\nabla w(t_n^-)}}{\tj{\sj{ v(t_n)}}}{\Gamma_n} 
- \lip{ \sj{w(t_n^-)}}{\tj{\sa{a\nabla  v(t_n)}}}{\Gamma_n}\\
&+ \lip{\sigma_0  \sj{w(t_n^-)}}{\tj{\sj{ v(t_n)}}}{\Gamma_n}\Big).
\end{split}
\end{equation}

It is now clear how to estimate most of the terms to obtain the stated result using the Cauchy-Schwarz inequality. 
The first two terms on the right hand side in the above sum are estimated as follows
\[
 \lip{ \sigma_1^{-1/2}(\sa{a\nabla\dot  w}- \sigma_0   \sj{\dot w})}{\sqrt{\sigma_1}\sj{ v}}{\Gamma_n \times I_n}
\leq \ltwo{\sigma_1^{-1/2}(\sa{a\nabla\dot  w}-  \sigma_0 \sj{\dot w})}{\Gamma_n \times I_n}
\ltwo{\sqrt{\sigma_1} \sj{v}}{\Gamma_n \times I_n};
\]
for the third term, we have 
\[
\begin{split}
\lip{ \sa{\dot w}}{ \sj{a\nabla v}}{\Gamma_n^{\rm int}\times I_n}   
& \le \ltwo{\sigma_2^{-1/2} \sa{\dot w}}{\Gamma_n^{\rm int} \times I_n}   \ltwo{\sqrt{\sigma_2} \sj{a\nabla v}}{\Gamma_n^{\rm int} \times I_n}.
\end{split}
\]
%making use of the inverse estimate \eqref{eq:disc_ineq}, and upon choosing $\rho = C_{\rho}p^2 h^{-1}$, with $C_{\rho}>0$ appropriate positive constant. Also we choose $16\tau \le \sigma$ (or better $C_{\sigma}$ large enough so that $16\tau \le \sigma$) and the proof follows.
\end{proof}

\begin{remark}\label{rem:cheaper}
 Note that \eqref{eq:trefftz_bfn} shows that for Trefftz functions the bilinear form can be evaluated without computing integrals over the volume terms $\Omega \times I_n$. This can bring considerable savings, especially in higher spatial dimensions; see Figure~\ref{fig:Conv_Time}.
\end{remark}

\reply{
It is possible to define the space-time dG method above with the classical (discontinuous) space of \emph{all} polynomials of degree $p$ (total degree or of degree $p$ on each variable). The resulting method then appears to work in practice also. Some numerical experiments and comparison with the smaller polynomial Trefftz space are given Section \ref{sec:numerics}. The error analysis of the method with the full polynomial space of degree $p$ is not straightforward, though. In particular, it is not immediately clear how to treat the volume terms which do not vanish in this case. This would be essential in completing an error analysis for such spaces also.}

\begin{theorem}\label{th:best_approx}
  Let $U \in \TspaceT$ be the discrete solution of the Trefftz time-space discontinuous Galerkin method and let \our{$u \in \solnspace$} be the exact solution. Then, we have
\[
%\dgnorm{U-u} \leq (  C_\star+1) \inf_{V \in \TspaceT}  \dgnorml{V - u}%+\dgnorm{V-u} \big)
\dgnorm{U-u} \leq\inf_{V \in \TspaceT}\big (  C_\star \dgnorml{V - u}+\dgnorm{V-u} \big). 
\]
%Further, if $\dspaceT{1} \subseteq \TspaceT$ then

%{\em \Large Missing estimate for linears}
\end{theorem}
\begin{proof}
By Galerkin orthogonality
\[
a(V -U, v) = a(V-u,v),
\]
for any $V,v \in \TspaceT$. Hence,
\[
\dgnorm{V - U}^2 = a(V - U, V - U)
 = a(V - u, V - U) \leq C_\star \dgnorml{V - u}\dgnorm{V - U},
\]
giving
\[
\dgnorm{U-u} \leq \dgnorm{V-U}+\dgnorm{V-u} \leq  C_\star\dgnorml{V - u}+\dgnorm{V-u}.
\]
%Standard error estimates complete the proof.
\end{proof}

To conclude this section we show that in the case of Trefftz polynomials, the discrete norm can be bounded below by an $L^2$-temporal norm.  For simplicity of the presentation \emph{only}, we shall, henceforth, make the following assumption.
 \our{
\begin{assumption} \label{as:st_reg}
We assume that $\diam(K\times I_n)/\rho_{K\times I_n} \leq c_{\mathcal T}$, for all $K \in \mathcal{T}_n$, $n = 0,1,\dots, N-1$.
\end{assumption}}

\our{In the simplifying context of space-time meshes consisting of prismatic meshes without local time-stepping, Assumption \ref{as:st_reg} implies global quasi-uniformity. Given the tensor-product/pris\-matic structure of the space-time elements $K\times I_n$, it is by all means possible to extend the error analysis below to space-time meshes not satisfying Assumption \ref{as:st_reg}. Moreover, with minor modifications only, it is possible to extend the findings of this work to meshes with space-time elements with variable temporal dimension lengths; the notational overhead for such a development is deemed excessive given the a-priori error analysis point of view of this work.}

\begin{proposition}\label{prop:positiveE}
  For any $v \in \Tspace$ it holds
\[
\ltwo{\dot v}{\Omega \times I_n}^2+\ltwo{\sqrt{a}\babla v}{\Omega \times I_n}^2 
\leq (t_{n+1}-t_n)e^{\widetilde C(t_{n+1}-t_n)/\underline{h}}\left(\ltwo{\dot v(t_n^+)}{ \Omega}^2+\ltwo{\sqrt{a} \babla v(t_n^+)}{ \Omega}^2 \right),
\]
where
$
\reply{\widetilde C} := c_{\mathcal T}C_{\text{inv}}p^2 C_a$ and
$ \underline{h} := \min_{x \in \Omega} h(x,t)$, $t \in I_n$.
The same estimate holds with $t_n^+$ replaced by $t_{n+1}^-$.

% Further, if $u$ solves the wave equation \eqref{eq:wave} 
% then
% \[
% \ltwo{\dot u}{\Omega \times I_n}^2+\ltwo{\sqrt{a}\babla u}{\Omega \times I_n}^2 
% = (t_{n+1}-t_n)\left( \ltwo{\dot u(t_n^+)}{ \Omega}^2 + \ltwo{\sqrt{a} \babla u(t_n^+)}{ \Omega}^2\right),
% \]
% where again the same  holds with $t_n^+$ replaced by $t_{n+1}^-$.
\reply{Consequently, under Assumption ~\ref{as:st_reg},
\[
\|\dot V\|^2_{\Omega \times (0,T)}+\|\sqrt{a}\babla V\|^2_{\Omega \times (0,T)}
\leq C e^{\widetilde C c_{\mathcal T}}T\dgnorm{V}^2,
\]
for all $V \in \TspaceT$ with a constant $C>0$ independent of the meshsize.}
\end{proposition}
\begin{proof}
Note that for an element $K$ with exterior normal $\nu$
\[
\begin{split}  
\frac{d}{dt}  \left(\tfrac12\ltwo{\dot v(t)}{K}^2+\tfrac12 \ltwo{\sqrt{a} \nabla v(t)}{K}^2\right)
 &= \lip{\ddot v(t)}{ \dot v(t)}{K} +  \lip{a\nabla v(t)}{ \nabla \dot v(t)}{K}\\
&= \lip{\nu \cdot a \nabla  v(t)}{\dot v(t)}{\partial K}\\
%&\leq \ltwo{\partial_\nu v(t)}{\partial K} \ltwo{\dot v(t)}{\partial K}\\
&\leq \tfrac12 \ltwo{\nu \cdot a \nabla  v(t)}{\partial K}^2 +\tfrac12 \ltwo{\dot v(t)}{\partial K}^2\\
&\leq  C_{\text{inv}}p^2 |\partial K|/|K|\left(\tfrac12\ltwo{a \nabla v(t)}{ K}^2+ \tfrac12\ltwo{\dot v(t) }{ K}^2 \right),\\
&\leq C_a C_{\text{inv}}p^2 c_{\mathcal T} h_K^{-1}\left(\tfrac12\ltwo{\sqrt{a}\nabla v(t)}{ K}^2+ \tfrac12\ltwo{\dot v(t) }{ K}^2 \right),
\end{split}
\]
where we have used the discrete trace inequality.
Gronwall inequality now gives us 
\[
\tfrac12\ltwo{\sqrt{a}\nabla v(t)}{K}^2+ \tfrac12\ltwo{\dot v(t)}{ K}^2 
\leq  e^{\widetilde C(t-t_n)/{h_K}}\left(\tfrac12\ltwo{\sqrt{a}\nabla v(t_n^+)}{ K}^2+ \tfrac12\ltwo{\dot v(t_n^+)}{ K}^2 \right),
\]
as well as
\[
\tfrac12\ltwo{\sqrt{a}\nabla v(t)}{K}^2+ \tfrac12\ltwo{\dot v(t)}{ K}^2 
\leq  e^{\widetilde C(t_{n+1}-t_n)/h_K}\left(\tfrac12\ltwo{\sqrt{a}\nabla v(t_{n+1}^-)}{ K}^2+ \tfrac12\ltwo{\dot v(t_{n+1}^-)}{ K}^2 \right),
\]
for all $t \in [t_n,t_{n+1}]$.  Integrating in time and summing over all $K$ gives the required result. \reply{The final inequality follows from the definition of the discrete dG norm.}
% \[
% \|\nabla u\|^2_{L^2( K \times I_n)}+ \|\dot u \|^2_{L^2( K \times I_n)} 
% \leq  \frac{h}C (e^{C (t_{n+1}-t_n)/h}-1)\left(\|\nabla u(t_n^+)\|^2_{L^2( K)}+ \|\dot u(t_n^+) \|^2_{L^2( K)} \right).
% \]
% Since
% \[
% a_n(v,v) = E(t_{n+1}^-,v)+E(t_{n}^+,v).
% \]
% the above lower bound follows
% Similarly 
% \[
% \|\nabla u(t)\|^2_{L^2( K)}+ \|\dot u \|^2_{L^2( K)} 
% \leq  e^{C(t_{n+1}-t_n)/h}\left(\|\nabla u(t_{n+1}^-)\|^2_{L^2( K)}+ \|\dot u(t_{n+1}^-) \|^2_{L^2( K)} \right).
% \]
\end{proof}

The above two results allow us to conclude that we can also bound the error in a more standard norm.
\begin{corollary}\label{cor:positiveE}
Under the hypothesis of Theorem~\ref{th:best_approx} and under Assumption \ref{as:st_reg}, we have
\[
\begin{split}  
\ltwo{\dot U -\dot u}{\Omega \times [0,T]}+\ltwo{\sqrt{a}\babla (U-u)}{\Omega \times [0,T]} 
\leq C &\inf_{V \in \TspaceT}\big (  C_\star \dgnorml{V - u}\\ &+\ltwo{\dot V -\dot u}{\Omega \times [0,T]}+\ltwo{\sqrt{a}\babla (V-u)}{\Omega \times [0,T]} \big),
\end{split}
\]
for some positive constant $C$, independent of the mesh parameters and of $u$ and $U$. 
\end{corollary}
\begin{proof}
\reply{
Triangle inequality and Proposition \ref{prop:positiveE} imply that, for any $V \in \TspaceT$, 
\[
\begin{split}  
& \ltwo{\dot U -\dot u}{\Omega \times [0,T]}+\ltwo{\sqrt{a}\babla (U-u)}{\Omega \times [0,T]}\\
\leq & \ltwo{\dot U -\dot V}{\Omega \times [0,T]}+\ltwo{\sqrt{a}\babla (U-V)}{\Omega \times [0,T]}
\\
&+ \ltwo{\dot V -\dot u}{\Omega \times [0,T]}+\ltwo{\sqrt{a}\babla (V-u)}{\Omega \times [0,T]}\\
\leq & C\dgnorm{U-V}
+ \ltwo{\dot V -\dot u}{\Omega \times [0,T]}+\ltwo{\sqrt{a}\babla (V-u)}{\Omega \times [0,T]}.
\end{split}
\]
As in the proof of Theorem~\ref{th:best_approx}, we also have that 
$
\dgnorm{V-U} \leq C_\star \dgnorml{V-u},
$
which completes the proof.
}
\end{proof}

\section{A priori error bounds}\label{apriori}
We shown next  the Trefftz basis 
is sufficient to deliver the expected rates of convergence for the proposed method.

\begin{lemma}\label{lemma:abstract_error_bound}
Let the setting of Theorem~\ref{th:best_approx} hold, let $\our{v}_h \in \TspaceT$ be an arbitrary function in the discrete space, and let $\eta = u-\our{v}_h$. Then
\begin{equation}\label{eq:abstract_error_bound}
\begin{aligned}
\dgnorm{U-u}^2 \leq \ &  C\sum_{n = 0}^{N-1} \sum_{K\in \mathcal{T}^n} 
\our{C_a}\bigg( \frac{p^2}{\our{\tau_n}}\Big( \our{\max\big\{1,\frac{\tau_n^2}{h_K^2}\big\}}\|\dot{\eta}\|_{ K\times I_n}^2+ \|\nabla\eta\|_{ K\times I_n}^2\Big)\\
&\qquad\quad\qquad+\our{\tau_n}\Big( \|\nabla\dot{\eta}\|_{ K\times I_n}^2+ \our{\max\big\{1,\frac{h_K^2}{\tau_n^2}\big\}}p^{-1}\|D^2\eta\|_{ K\times I_n}^2\Big)\\
&\qquad\quad\qquad+\frac{\our{h_K^2\tau_n}}{p^4} \|D^2\dot{\eta}\|_{ K\times I_n}^2+\our{\frac{p^4}{h_K^2\tau_n} \|\eta\|_{ K\times I_n}^2}\bigg),
\end{aligned}
\end{equation}
where $U \in \TspaceT$ is the discrete solution. 
% and $\alpha:\bar{\Omega}\to\mathbb{R}_+$, with
% \[
% \alpha|_K: =\max_{K':\bar{K}\cap\bar{K}'\neq \emptyset} \alpha_K,\qquad \alpha_K:=\max_{x\in \bar{K}} a(x).
%\]
\end{lemma}
\begin{proof}
 Theorem \ref{th:best_approx} implies 
\[
\dgnorm{U-u} \leq  C_\star \dgnorml{\eta}+\dgnorm{\eta}.
\]
We shall now estimate each term of the norms on the right-hand side. We shall repeatedly use the standard trace estimate 
$
 \| v \|_{\partial \omega}^2 \le C \| v \|_{ \omega} ( \| v \|_{ \omega}^2+ \| \nabla v \|_{ \omega}^2)^{1/2},
$
for $v\in H^1(\omega)$, where $\omega$ is a subset of \reply{$\mathbb{R}^k$, $k=1,\dots,d+1$}.
We proceed as follows:
\[
\begin{aligned}
\sum_{n = 1}^{N} \|\dot \eta(t_n^-)\|^2_{\Omega} 
&=   \sum_{n = 1}^{N} \sum_{K\in \mathcal{T}^{n-1}} \|\dot \eta\|^2_{K\times \{t_n^-\}}
\le C\sum_{n = 0}^{N-1} \sum_{K\in \mathcal{T}^n} \Big( \frac{C_a p}{\our{\tau_n} }\|\dot \eta\|^2_{K\times I_n}+ \frac{\our{\tau_n} }{C_a p}\|\ddot \eta\|^2_{K\times I_n}\Big)\\
& \le C\sum_{n = 0}^{N-1} \sum_{K\in \mathcal{T}^n} \Big( \frac{C_a p}{\our{\tau_n} }\|\dot \eta\|^2_{K\times I_n}+ \frac{\our{\tau_n}}{C_a p}\|\nabla\cdot a(\cdot)\nabla \eta\|^2_{K\times I_n}\Big)\\
& \le C\sum_{n = 0}^{N-1} \sum_{K\in \mathcal{T}^n} C_a \Big( \frac{p}{\our{\tau_n} }\|\dot \eta\|^2_{K\times I_n}+ \frac{\our{\tau_n}}{p}\|\Delta \eta\|^2_{K\times I_n}\Big);
\end{aligned}
\]
We prefer to retain an explicit dependence on the polynomial degree $p$ at this point, as it will be of relevance in the error analysis for $d=1$.
In analogous fashion, we also have
\[
\begin{aligned}
\sum_{n = 1}^{N} \| \sqrt{a}\nabla \eta(t_n^-)\|^2_{\Omega} 
&
\le C\sum_{n = 0}^{N-1} \sum_{K\in \mathcal{T}^n} C_a \Big( \frac{p}{\our{\tau_n}}\|\nabla \eta\|^2_{K\times I_n}+ \frac{\our{\tau_n}}{ p}\|\nabla\dot \eta\|^2_{K\times I_n}\Big).
\end{aligned}
\]
Next, we estimate the penalty term:
\[
\begin{aligned}
\sum_{n = 1}^{N}  \|\sqrt{\sigma_0^{}}\sj{\eta(t_n^-)}\|^2_{\Gamma_n} 
&
\le C  \sum_{n = 0}^{N-1} \sum_{K\in \mathcal{T}^n} C_a \Big( \frac{p^3}{\our{\tau_n h_K}}\| \eta\|^2_{\partial K\times I_n}+ p\our{\frac{\tau_n}{h_K}}\|\dot \eta\|^2_{\partial K\times I_n}\Big)\\
&
\le C  \sum_{n = 0}^{N-1} \sum_{K\in \mathcal{T}^n} C_a \Big( \frac{p^4}{\our{\tau_n h_K^2}}\| \eta\|^2_{K\times I_n}+\frac{p^2}{\our{\tau_n} }\|\nabla \eta\|^2_{K\times I_n}\\
   &\qquad \qquad\qquad\quad+ \frac{p^2\our{\tau_n}}{\our{h_K^2}}\|\dot \eta\|^2_{K\times I_n}+ \our{\tau_n} \| \nabla\dot\eta\|^2_{K\times I_n}\Big).
\end{aligned}
\]
Similarly, we also have
\[
\begin{aligned}
\sum_{n = 1}^{N}  \|\sigma_0^{-1/2}\sa{a\nabla \eta(t_n^-)}\|^2_{\Gamma_n}
&
\le C  \sum_{n = 0}^{N-1} \sum_{K\in \mathcal{T}^n} C_a \Big( \frac{\our{h_K}}{p\,\our{\tau_n}}\| \nabla \eta\|^2_{\partial K\times I_n}+ \frac{\our{h_K\tau_n}}{p^3}\|\nabla\dot \eta\|^2_{\partial K\times I_n}\Big)\\
&
\le C  \sum_{n = 0}^{N-1} \sum_{K\in \mathcal{T}^n} C_a \Big( \frac{1}{\our{\tau_n} }\| \nabla \eta\|^2_{K\times I_n}+\frac{\our{h_K^2} }{p^2\our{\tau_n}}\|D^2 \eta\|^2_{K\times I_n}\\
   &\qquad \qquad\qquad\quad+ \frac{\our{\tau_n}  }{p^2}\|\nabla\dot \eta\|^2_{K\times I_n}+\frac{\our{h_K^2\tau_n} }{p^4}\| D^2\dot\eta\|^2_{K\times I_n}\Big).
\end{aligned}
\]
Next, recalling that $\sigma_1|_{\partial K\cap \Gamma_n\times I_n}= C_a p^3/(\our{h_K\tau_n} )$, we estimate
\[
\begin{aligned}
\sum_{n = 0}^{N-1}\ltwo{\sqrt{\sigma_1}\sj{\eta}}{\Gamma_n \times I_n}^2
&
\le C  \sum_{n = 0}^{N-1} \sum_{K\in \mathcal{T}^n} C_a \Big( \frac{p^4}{\our{\tau_n h_K^2}}\|  \eta\|^2_{ K\times I_n}+ \frac{p^2}{\our{\tau_n} }\|\nabla \eta\|^2_{K\times I_n}\Big).
\end{aligned}
\]
Further, since $\sigma_2= \our{h_K}/(C_a\our{\tau_n})$,  we have
\[
\begin{aligned}
\sum_{n = 0}^{N-1} \ltwo{\sqrt{\sigma_2}\sj{a\nabla \eta}}{\Gamma_n \times I_n}^2
&
\le C  \sum_{n = 0}^{N-1} \sum_{K\in \mathcal{T}^n} C_a \Big( \frac{p^2}{\our{\tau_n} }\|\nabla  \eta\|^2_{ K\times I_n}+ \frac{\our{h_K^2}}{p^2\our{\tau_n}}\|D^2 \eta\|^2_{K\times I_n}\Big).
\end{aligned}
\]
%this arises again from the trace estimate \| v \|_{\partial K}^2 \le C (  d^{-1} \| v \|_{ K}^2 +d \| \nabla v \|_{ K}^2 but with d= \sqrt{h/p^2}.
The next term is treated as follows:
\[
\begin{aligned}
\sum_{n = 0}^{N-1} \ltwo{\sigma_2^{-1/2}\sa{\dot \eta}}{\Gamma_n^{\rm int} \times I_n}^2
&
\le C  \sum_{n = 0}^{N-1} \sum_{K\in \mathcal{T}^n} C_a \Big( \frac{p^2\our{\tau_n}}{\our{h_K^2} }\| \dot \eta\|^2_{ K\times I_n}+ \frac{\our{\tau_n} }{p^2}\|\nabla \dot \eta\|^2_{K\times I_n}\Big).
\end{aligned}
\]
Continuing, we have
\[\begin{aligned}
\sum_{n = 0}^{N-1}\ltwo{\sigma_1^{-1/2} \sa{a\nabla \dot \eta}}{\Gamma_n \times I_n}^2
&
\le C  \sum_{n = 0}^{N-1} \sum_{K\in \mathcal{T}^n} C_a \Big( \frac{\our{\tau_n} }{p^2}\| \nabla \dot\eta\|^2_{ K\times I_n}+ \frac{\our{\tau_n h_K^2}}{p^4}\| D^2 \dot \eta\|^2_{K\times I_n}\Big).
\end{aligned}
\]
Finally, we estimate
\[
\begin{aligned}
\sum_{n = 0}^{N-1}\ltwo{\sigma_0\sigma_1^{-1/2}\sj{\dot u}}{\Gamma_n \times I_n}^2
&
\le C  \sum_{n = 0}^{N-1} \sum_{K\in \mathcal{T}^n} C_a \Big( \frac{p^2\our{\tau_n}}{\our{h_K^2} }\|  \dot\eta\|^2_{ K\times I_n}+ \our{\tau_n} \| \nabla \dot \eta\|^2_{K\times I_n}\Big).
\end{aligned}
\]
%\sigma_0^2\sigma_1^{-1} = (C_a^2 p^4/h_K^{(n)} ^2 ) h_K^{(n)} ^2 / (C_a p^3) =C_a p
The remaining terms in $\dgnorm{\eta}$ are treated completely analogously.
\end{proof}

\our{We have kept the explicit dependence on the local spatial and temporal meshsizes to further emphasize that the proposed method does not require any CFL-type restrictions for stability and convergence: as long as $h_K\sim \tau_n$ the above bound is sufficient to show optimal convergence, as we shall see below.}
To complete the error analysis, we need to prove the existence of an appropriate approximation in $\TspaceT$ of the exact solution. We first show how to obtain
such an approximation locally.

\begin{proposition}
Let $J \subset \mathbb{R}^{d+1}$ be star-shaped with respect to a ball $B \subset J$. Then, there exists a projector 
\[
\Pi^{p}: H^{p+1}(J) \rightarrow \mathcal{P}_p(J)
\]
such that for any $v \in H^{p+1}(J)$
\[
\|D^\beta (v-\Pi^{p}v)\|_{J} \leq  C(\diam(J))^{p+1-|\beta|} \|v\|_{H^{p+1}(J)}, \qquad
|\beta| \leq p
\]
and further if $v$ satisfies the wave equation $\ddot u - \nabla \cdot a \nabla v = 0$ in $J$ then so does $\Pi^pv$. The constant $C$ depends on $p$ and on the shape of $J$\footnote{By the shape of the domain, is here meant the chunkiness parameter $\diam(J)/\rho_{\text{max}}$ where $\rho_{\text{max}} = \sup\{\rho : J \text{ is star-shaped with respect to a ball of radius } \rho\}$}.
\end{proposition}
\begin{proof}
We can define $\Pi^p v$ to be the averaged Taylor polynomial of order $p$,
\[
\Pi^p v(x) = \sum_{|\alpha| \leq p} \frac1{\alpha!} \int_B D^\alpha v(y) (x-y)^\alpha \phi(y) dy,
\]
where $\phi \in C_0^\infty(\mathbb{R}^{d+1})$ is an arbitrary cut-off function satisfying $\int_B \phi = 1$ and $\operatorname{supp} \phi = \overline{B}$; for details see \cite{brenner_scott}. Then Bramble-Hilbert lemma gives us the approximation property required, see \cite[Lemma~4.3.8]{brenner_scott}, and so it only remains to show that $\Pi^p v$ satisfies the wave equation if $v$ does. For $p \leq 1$ this is clear. For $p \geq 2$, the result follows from the linearity of $\Pi^p$, the fact that the wave equation contains only operators of second order, and the following property of averaged Taylor polynomials:
\[
D^\alpha \Pi^p v = \Pi^{p-|\alpha|} D^\alpha v, \qquad |\alpha| \leq p.
\]
\end{proof}

Applying such a projector to the exact solution and combining this with Lemma~\ref{lemma:abstract_error_bound} gives us a proof of the convergence order of the discrete scheme.

\begin{theorem}\label{thm:conv_order}
  Let the exact solution $u \in \solnspace$ be such that for each space time element $K \times I_n$, $u|_{K \times I_n} \in H^{s+1}(K \times I_n)$ for some $0 \leq s \leq p$. 
 Then
\begin{equation}\label{eq:gen_error_bound}
\dgnorm{U-u} \leq C\left( \sum_{n = 0}^{N-1} \sum_{K\in \mathcal{T}^n} (h_K^{(n)})^{2s-1}  \|u\|^2_{H^{s+1}(K \times I_n)}\right)^{1/2} \leq C(\reply{u}) h^{s-1/2},
\end{equation}
where $h = \max_{K,n} h^{(n)}_K$ and 
\[
C(\reply{u}) = \left(\sum_{n = 0}^{N-1} \sum_{K\in \mathcal{T}^n}  \|u\|^2_{H^{s+1}(K \times I_n)}\right)^{1/2}.
\]
\end{theorem}

%Therefore, the validity of the bound \eqref{eq:gen_error_bound} rests upon a proof of approximation estimates of the form \eqref{eq:approx_bound_gen} for polynomial Trefftz basis. We have already proved the existence of such a $\Pi^{h,p}u$ in any dimension if $u$ is locally continuously differentiable. In dimension $d = 1$, we shall show $hp$-version a priori bounds and give a proof of exponential convergence of the $p$-version space-time dG method for the wave problem.

% \subsection{Approximation properties of the multidimensional Trefftz basis}\label{sec:approx}

% If we assume that locally $v$ has a Taylor expansion of sufficiently \reply{high degree}, then the projection operator in \eqref{eq:approx_bound_gen} can be defined as the Taylor polynomial; see \cite{MacJ} and \cite{Whi03}. The following lemma and corollary give the details.

\subsection{$hp$-version error analysis for $d=1$}\label{sec:approx_one}
As we shall now show, the Trefftz basis is sufficient to deliver the expected $hp$-version a priori error bounds for $d=1$, along with a proof of the exponential convergence of the $p$-version space-time dG method for the case of analytic exact solutions. \reply{See \cite{KreMPS} for another $hp$-analysis in 1D of a different Trefftz based method.}

To discuss the Trefftz-basis case for $d=1$, we let $K = [x_0,x_1]$ and start from the basic observation that the exact solution to the wave equation on each space time element is of the form
\begin{equation}\label{eq:exact_sol_form}
u(x,t)|_{K\times I_n} = F^1_{n,K}(a^{-1/2}x+t)+F^2_{n,K}(a^{-1/2}x-t),
\end{equation}
where we can define $F^1$ and $F^2$ by
\[
F^1_{n,K}(a^{-1/2}x+t) = \tfrac12 u(x,t)+\tfrac12v(x,t), \qquad
(x,t) \in K \times I_n
\]
and
\[
F^2_{n,K}(a^{-1/2}x-t) = \tfrac12 u(x,t)-\tfrac12 v(x,t), \qquad
(x,t) \in K \times I_n,
\]
where
\[
v(x,t) = a^{1/2}\int_{t_n}^t u_x(x,\tau) d\tau+a^{-1/2}\int_{x_0}^xu_t(x',t_n)dx'.
\]
It is not difficult to see that these are well-defined, i.e., that the right-hand sides indeed depend only on $a^{-1/2}x\pm t$ by virtue of satisfying the equations $a^{1/2} f_x \mp f_t = 0$ respectively.
%for some sufficiently functions $F^i_{n,K}$, $i=1,2$, $K\in\mathcal{T}_n$; the smoothness of $F^i_{n,K}$ depends directly on the smoothness of the initial conditions.

For $\hat{I}:=(-1,1)$, we define
% the (generic) univariate projection operator $\hat{\lambda}_p:H^m(\hat{I})\to
%\mathcal{P}_p(\hat{I})$, $s>0$, $p\ge 1$, $\mathcal{P}_p(\hat{I})$ being
%the space of polynomials of degree $p$ or less on $\hat{I}$. Two specific choices for  $\hat{\lambda}_p$ will be used below: the classical $hp$-projection operator \cite{canuto_quarteroni,babuska_suri} for which we have
%\[
%\ltwo{(u-\hat{\lambda}_p u)^{(m)} }{I}\leq C p^{-r+m} \ltwo{u^{(r)}}{I},
%\]
%for all $m,r\in \mathbb{R}_+$, $m\le r$ and $C$ dependent on $m,r$ but not on $p$ or $u$, and 
the $H^1$-projection operator $\hat{\lambda}_p:H^1(\hat{I})\to
\mathcal{P}_p(\hat{I})$, $p\ge 1$, defined by
setting, for $\hat{u}\in H^1(\hat{I})$,
\[(\hat{\lambda}_p\hat{u})(x):=\int_{-1}^{x}\hat{\pi}_{p-1}(\hat{u}')(\eta)\ud
\eta +\hat{u}(-1),\qquad x\in \hat{I},\] with $\hat{\pi}_{p-1}$
being the $L^2$-orthogonal projection operator onto
$\mathcal{P}_{p-1}(\hat{I})$. 

Now, upon considering the linear scalings $\psi^1_{n,K}:\hat{I}\to J^1_{n,K}$,  $K\in\mathcal{T}_{n}$, such that
\[
J^1_{n,K}:=( \min_{(x,t)\in K\times I_n} \{x+ct\}, \max_{(x,t)\in K\times I_n} \{x+ct\} ),
\]
and $\psi^2_{n,K}:\hat{I}\to J^2_{n,K}$,  $K\in\mathcal{T}_{n}$, such that
\[
J^2_{n,K}:=( \min_{(x,t)\in K\times I_n} \{x-ct\}, \max_{(x,t)\in K\times I_n} \{x-ct\} ),
\]
we define the univariate space-time elemental projection operators $\lambda^i_p$, $i=1,2$, piecewise by
\[
(\lambda^i_p F)|_{J^i_{n,K}} := \hat{\lambda}^i_p ((F\circ \psi^i_{n,K})|_{\hat{I}}), \quad K\in\mathcal{T}_{n},\ n=0,1,\dots, N-1.
\]
Using these, we can now define the \emph{Trefftz projection} $\Pi_p u$ of a function $u$ of the form \eqref{eq:exact_sol_form} element-wise by
\begin{equation}\label{trefftz_proj}
(\Pi_p u)|_{K\times I_n} := \lambda^1_p F^1_{n,K}(x+ct) + \lambda^2_p F^2_{n,K}(x-ct),
\end{equation}
$K\in\mathcal{T}_{n}$, $n=0,1,\dots, N-1$. The approximation properties of $\Pi_p$ follow from the respective properties of $\lambda^i_p$, $i=1,2$. Space-time shape regularity implies $J^i_{n,K}\sim h_K^{(n)}$, $i=1,2$. 

We denote by $\Phi(p,s)$ the quantity
$
 \Phi(p,s):=(\Gamma(p-s+1)/\Gamma(p+s+1))^{\ha},
$
with $p,s$ real numbers such that $0\le s\le p$
and $\Gamma(\cdot)$ being the Gamma function; we also adopt the standard convention
$\Gamma(1)=0!=1$. Making use of \emph{Stirling's formula},
$\Gamma(n)\sim\sqrt{2\pi}n^{n-\ha}e^{-n}$, $n>0$,
we have, 
$
  \Phi(p,s) \le C p^{-s}
$, for $p\ge 1$,
with $0\le s\le p$ and $C>0$ constant depending only on $s$.

We have the following $hp$-approximation results for $\lambda^i_p$, $i=1,2$.

\begin{lemma}\label{approx_theorem}
Let $v \in H^{k+1}(J)$, for $k\ge 1$, and let $h=\diam(J)$ for \reply{an open, bounded interval $J\subset\mathbb{R}$}; finally let $\lambda_p$ be any of the $\lambda^i_p$, $i=1,2$. Then the following error bounds hold:

\begin{equation}\label{ltwobounis}
   \ltwo{v - \lambda_p v}{J} \le Cp^{-1}\Phi(p,s)h^{s+1}|v|_{s+1,J},
\end{equation}
and
\begin{equation}\label{dxbounis}
   \ltwo{v' - (\lambda_p v)'}{J} \le C\Phi(p,s)h^{s}|v|_{s+1,J},
\end{equation}
with $0\le s\le \min\{p,k\}$, $p\ge 1$. 
%Also, we have
%\begin{equation}\label{l2boundary2stais}
%   \ltwo{v- \Pi_p v}{\partial J}  \le Cp^{-1/2}\Phi(p,t)h^{t+1/2}|u|_{t+1,J},
%\end{equation}
%with $1\le t\le \min\{p,k\}$, $p\ge 1$. 
%Also, for $v\in H^{k+1}(\kappa)$, with $k\ge 2$, the following error estimate holds:
%\begin{equation}\label{ioddxstais}
%  \ltwo{v'-(\Pi_p v)'}{\partial J}  \le C p^{1/2}\Phi(p,l)h^{l-1/2}|u|_{l+1,J},
%\end{equation}
%with $2\le l\le\min\{p,k\}$, $p\ge 2$. 
Also, let $v\in H^{k+1}(J)$, with $k\ge 2$. Then, the following bound holds:
\begin{equation}\label{htwo_bound}
  \ltwo{v''- (\lambda_p v)''}{J}\le C p^{3/2}\Phi(p,m)h^{m-1}|v|_{m+1,J},
\end{equation}
with $1\le m\le \min\{p,k\}$. Finally, let $v\in H^{k+1}(J)$, with $k\ge 3$. Then, the following bound holds:
\begin{equation}\label{hthree_bound}
  \ltwo{v'''- (\lambda_p v)'''}{J}\le C p^{7/2}\Phi(p,l)h^{l-2}|v|_{l+1,J},
\end{equation}
with $2\le l\le \min\{p-1,k\}$.
\end{lemma}
\begin{proof}
 The proof 
of (\ref{ltwobounis}) and (\ref{dxbounis}) for the $H^1$-projection $\lambda_p^i$ can be found, e.g., in \cite{SchwabBook}. 
%The proof of
%(\ref{l2boundary2stais}) follows from (\ref{ltwobounis}), (\ref{dxbounis}) along with the standard trace
%inequality. 
%A different proof of (\ref{l2boundary2stais}) and the proof of (\ref{ioddxstais})
%can be found in \cite{georgoulis-suli:augmented}. 
The proof of \eqref{htwo_bound} can be found in \cite{GH}, while the proof of \eqref{hthree_bound} follows along the same lines as in the proof of  \eqref{htwo_bound} from \cite{GH}.
\end{proof}

These $hp$-approximation estimates imply the following bound.
 
\begin{theorem}\label{approx_theorem_dg_norms}
Let $u|_{K\times I_n} \in H^{k+1}(K\times I_n)$, for $k\ge 3$ be the exact solution to \eqref{eq:wave}. Then, \our{for space-time meshes satisfying Assumption \ref{as:st_reg},} the following error
bounds hold:
\begin{equation}\label{eq:aprioribound_hp}
\dgnorm{U-u}^2 \leq \reply{C}  p^{3} \Phi^2 (p,s)\sum_{n = 0}^{N-1} \sum_{K\in \mathcal{T}^n} \our{ \diam(K\times I_n)}^{2s-1} |u|_{s+1,K\times I_n}^2,
\end{equation}
for $3\le s\le \min\reply{\{}p+1,k\reply{\}}$ and $h=\max_{K,n}\{ \our{ \diam(K\times I_n)}\}$, with $C>0$ constant, independent of $p$, $h$, \reply{$u$, and of} $U$. Moreover, if $u$ is analytic on a neighbourhood of $\Omega$, there exists $r>0$, depending on the analyticity region of $u$ \reply{in a neighbourhood of $\Omega\times(0,T)$}, such that
\begin{equation}\label{eq:aprioribound_hp_exponential}
\dgnorm{U-u}^2 \leq  C(u) p^{3} \exp(-r p)\sum_{n = 0}^{N-1} \sum_{K\in \mathcal{T}^n}  |K\times I_n|  \our{\diam(K\times I_n)}^{2s-1}.
\end{equation} 
\end{theorem}

\begin{proof}
The proof of \eqref{eq:aprioribound_hp} follows by combining the $hp$-approximation bounds from \eqref{approx_theorem} with Lemma \ref{lemma:abstract_error_bound}.

For \eqref{eq:aprioribound_hp_exponential}, we work as follows. Analyticity of $u$ implies that there exists a $d>0$, \reply{such that for all $s\ge 0$,}
\begin{equation}\label{analyticity}
|u|_{s,K\times I_n} \le C d^s\Gamma(s+1)|K\times I_n|^{1/2},
\end{equation}
\reply{cf., e.g., \cite[Theorem 1.9.3]{davis}}.
Using this \eqref{analyticity}, setting $s=\reply{\gamma} p$ for some $0<\reply{\gamma}<1$, along with Stirling's formula, we can arrive to the bound
\[
\Phi^2 (p,\gamma p) |u|_{\gamma p+1,K\times I_n}^2 \le C \bigg( (2\gamma d)^{2\gamma}\frac{(1-\gamma)^{1-\gamma}}{(1+\gamma)^{1+\gamma}} \bigg)^p|K\times I_n|,
\]
with the precise choice of $\gamma$ remaining at our disposal. The function 
\[
F(\gamma) := (2\gamma d)^{2\gamma}\frac{(1-\gamma)^{1-\gamma}}{(1+\gamma)^{1+\gamma}} ,
\]
has a minimum at $\gamma_{\min}:=(1+4d^2)^{-1/2}$, giving $F(\gamma_{\min})<1$. Setting, now $r = 1/2|\log F(\gamma_{\min})|$, the result follows.
\end{proof}

\begin{remark}
The bound \eqref{eq:aprioribound_hp} is suboptimal in $p$ by one order. This is a standard feature of $hp$-version dG methods whose analysis requires the use of $hp$-type inverse estimates. It is possible to slightly improve on this result and obtaining only $1/2$ order $p$-suboptimal bounds, using the classical $hp$-approximation results from \cite{Can_Quar,BS}, instead of the $H^1$-projection operator as done above. These results, however, are not suitable for the proof of the exponential rate of $p$-convergence. 
\end{remark}

\section{Numerical experiments}\label{sec:numerics} \reply{We present a series of numerical experiments aiming to highlight the performance of the proposed method above. In each experiment, the spatial meshes are kept fixed $\mathcal{T} = \mathcal{T}_n$ and uniform time-step is used. In the one-dimensional ($d=1$) examples the spatial mesh is a uniform set of intervals, whereas for $d=2$, the spatial mesh is a quasiuniform triangulation. The resulting linear systems at each timestep are solved by standard sparse direct solvers.}
\subsection{Experiments in one dimension}
We consider the wave equation with constant diffusion coefficient $a\equiv 1$, spatial domain $\Omega = (0,1)$  and initial data 

\begin{equation}
  \label{eq:init_data}  
u(x,0) = e^{-\left(\frac{x-5/8}{\delta}\right)^2}, \qquad \dot u(x,0) = 0,
\end{equation}
where $\delta \leq \delta_0 = 7.5 \times 10^{-2}$. Note that the initial data is not exactly zero at the boundary, but is less than $10^{-11}$ in the range of parameter $\delta$ that we consider. This slight incompatibility with the boundary condition does not influence in any visible way our numerical results. Since the energy of the exact solution stays constant it is given for all times by
\[
\text{exact energy} = \tfrac12\ltwo{u_x(x,0)}{\Omega}^2 \approx 2 \delta^{-1}\int_{-\infty}^\infty y^2 e^{-2y^2}dy = \delta^{-1} \frac{\sqrt{\pi}}{2\sqrt{2}},
\]
where the approximation in the second step is of the order of $10^{-11}$ for reasons given above and the final equality is obtained by using integration by parts to reduce it to the Gaussian integral \cite{GraR2000}.

As the problem is in one spatial dimension, the exact solution is not difficult to obtain. The error will be computed in the dG norm 
\[
\text{error } = \dgnorm{u- u_h}.
\]
Since the exact solution is smooth, note that, see \eqref{eq:dg_norm_expr}, 
\[
\dgnorm{u}^2 = a(u,u) = 2 \times \text{exact energy}.
\]

\subsubsection{Convergence order}

We first investigate the convergence order of the numerical method. Though we did not analysed full polynomial spaces,  we also give numerical experiments for these as it is interesting to compare the two sets of results.

In this subsection, we choose $\delta_0 = 7.5 \times 10^{-2}$ and $T = 1/4$. Note that we choose such a small time interval in order to reach the asymptotic regime earlier -- this will especially be important for lower orders. In Figure~\ref{fig:conv} and Tables~\ref{tab:conv_order} and \ref{tab:linear_order}, the convergence curves and numerically computed convergence orders are given. These confirm the theoretical results. Note that the errors obtained by the full and the Trefftz spaces are very similar for the same order, but the Trefftz spaces require fewer degrees of freedom and cheaper implementation; see Remark~\ref{rem:cheaper} and Figure~\ref{fig:Conv_Time}. We have also found that higher order approximation converges without the two extra stabilisation terms, i.e., with $\sigma_1 = \sigma_2 = 0$, but with the piecewise linear functions it stagnates.

\begin{figure}[hbtp]
 \centering
\includegraphics[scale=0.45]{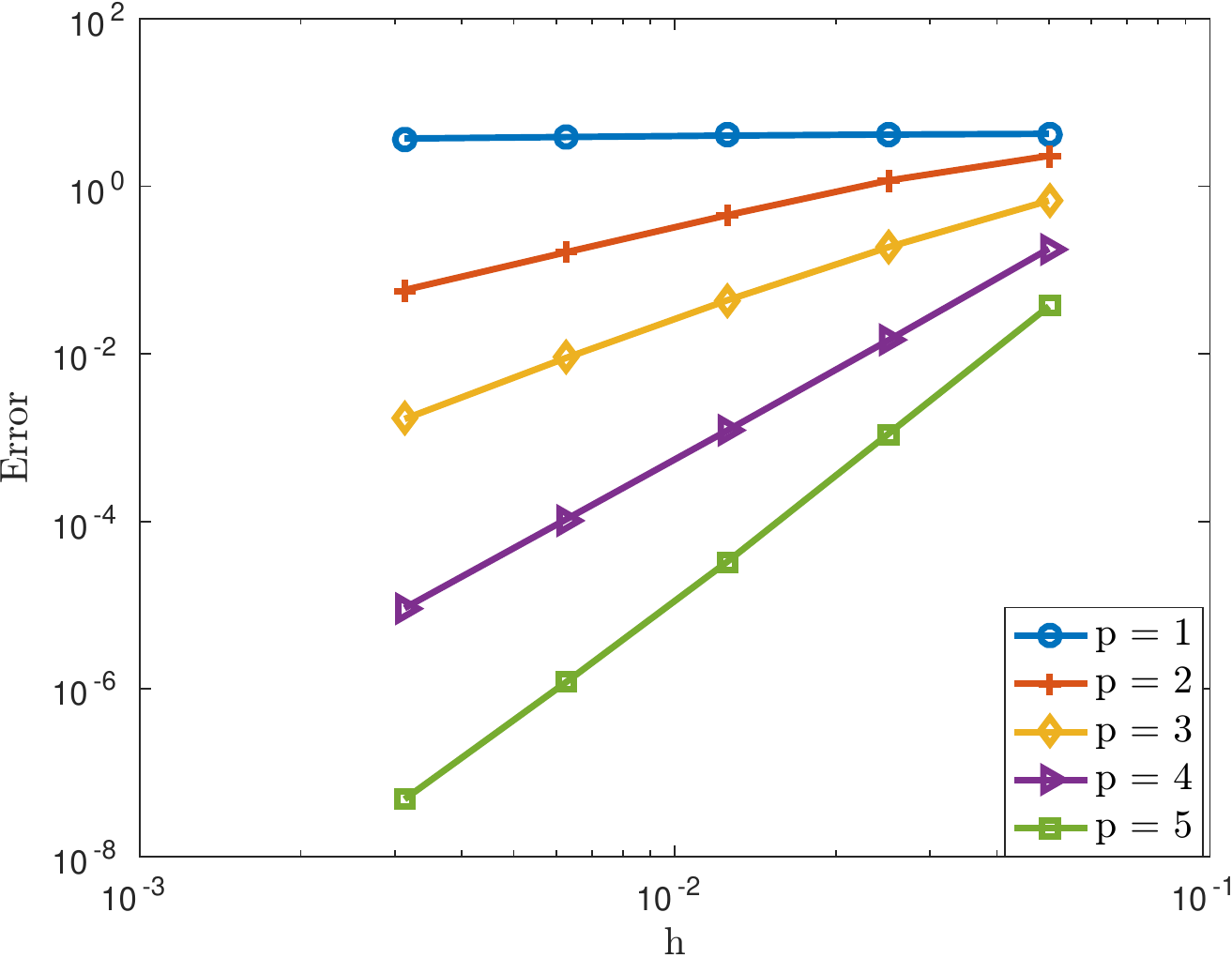}  
\includegraphics[scale=.45]{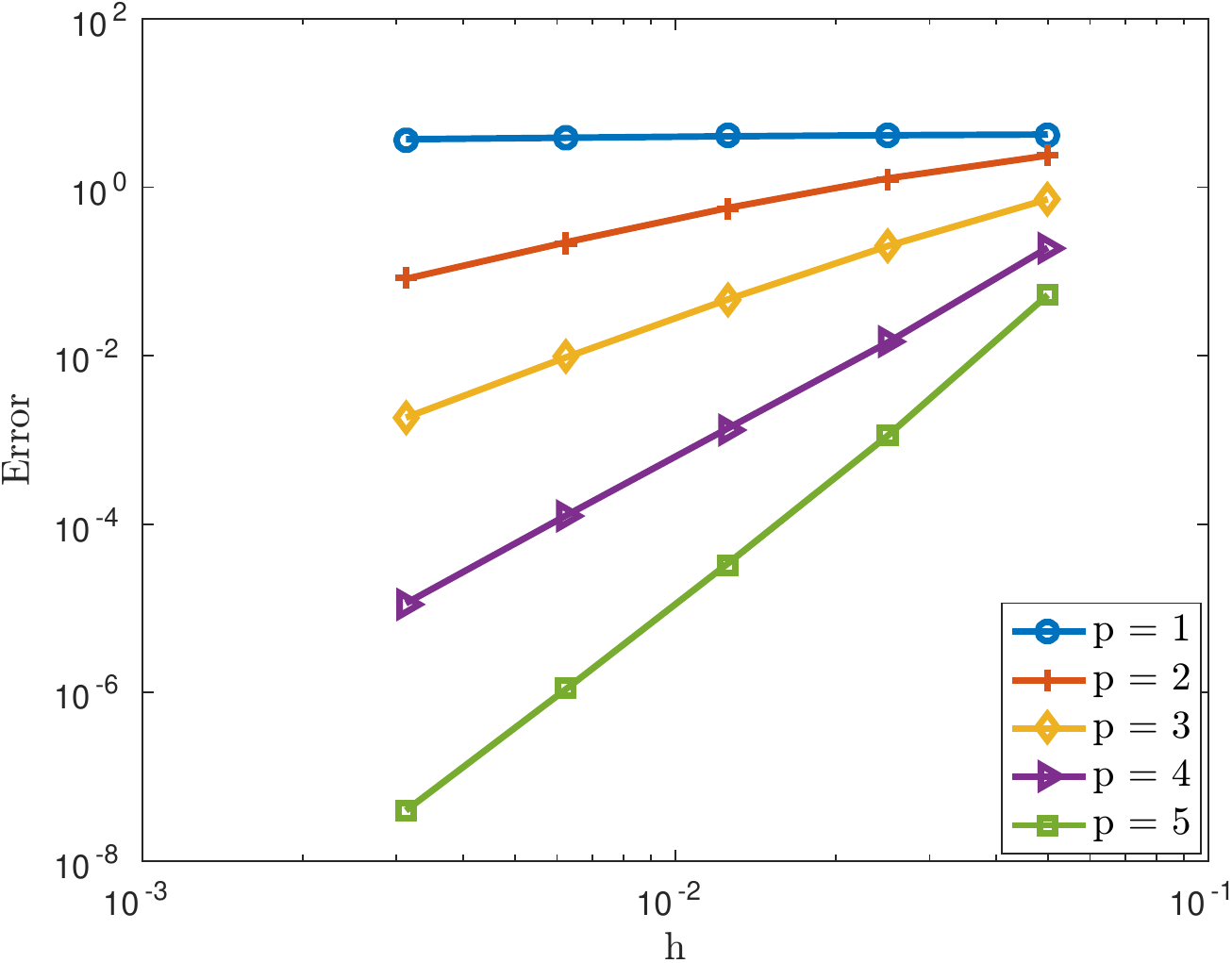} 
\caption{\small Convergence of the Trefftz, on the left, and polynomial, on the right, time-space dG method of order $p$. The error is plotted  against the uniform mesh width in time and space $h = T/N$. }
 \label{fig:conv}
\end{figure}
%%%% Time up to oder 5
\begin{figure}[hbtp]
\centering
\includegraphics[scale=0.450]{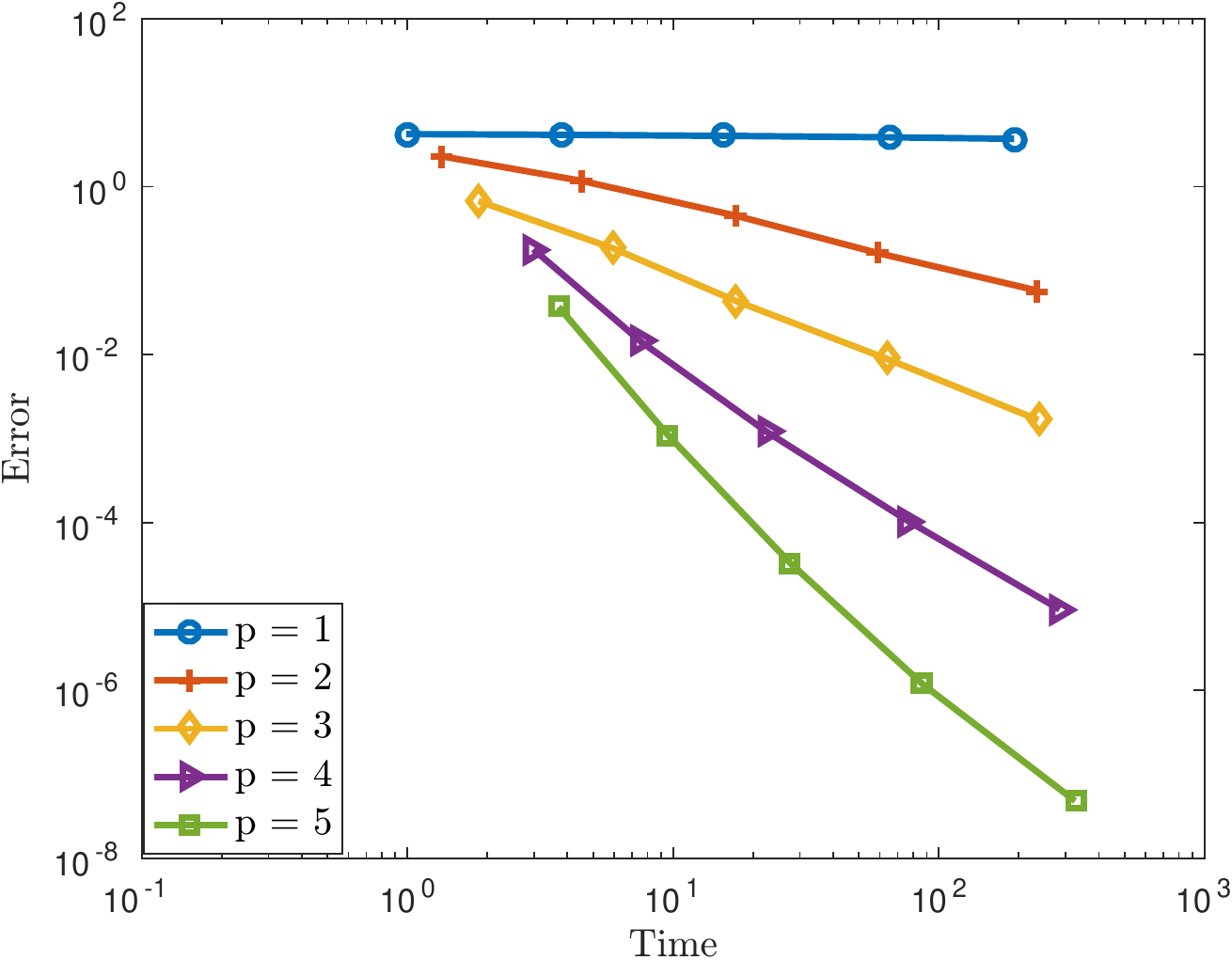}
\includegraphics[scale=0.450]{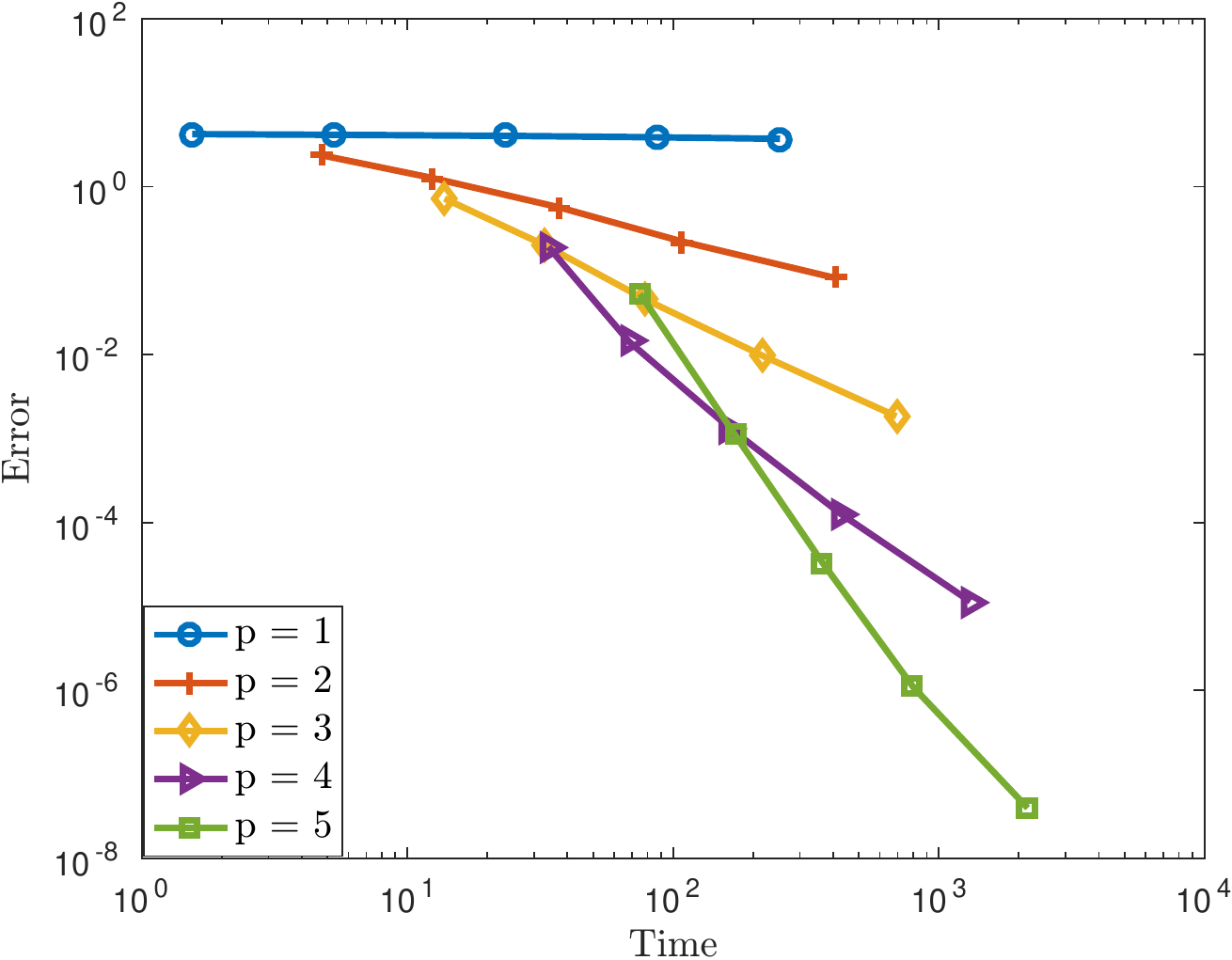} 
\caption{\small Error against computational time for Trefftz (on the left) and polynomial spaces (on the right). The much lower times for Trefftz spaces are due both to the smaller number of degrees of freedom for the same accuracy and to the cheaper construction of matrices; see Remark~\ref{rem:cheaper}.}
\label{fig:Conv_Time}
\end{figure}

%%%%%%%%%%%%%%%%%%%%%%%%%%%%%%%%%%%%%%%%%%%%%

\begin{figure}[hbtp]
\centering
\includegraphics[scale=.4]{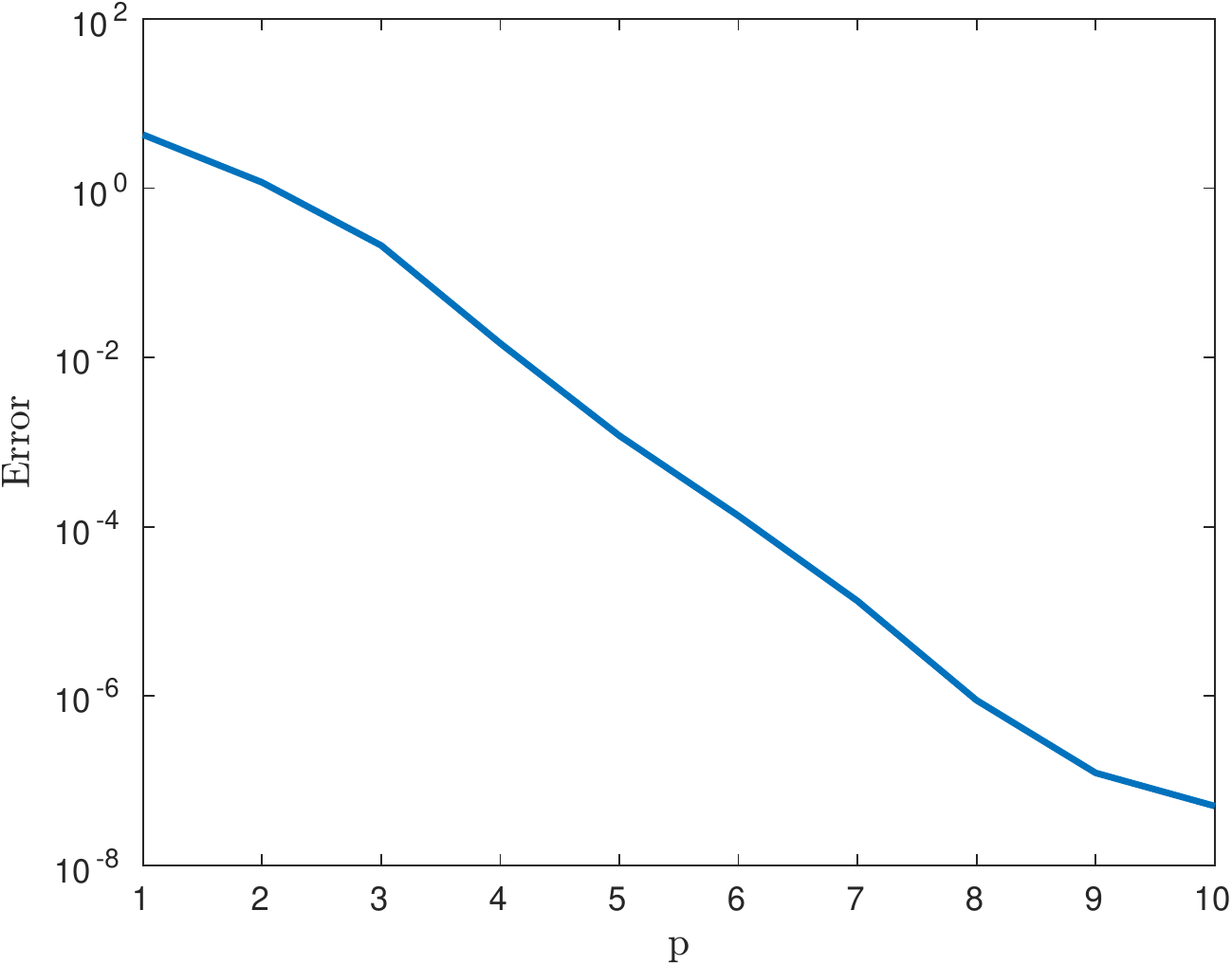} 
\caption{\small Convergence of the Trefftz method with fixed mesh width $h=1/40$ and increasing polynomial order $p$.}
\label{fig:Convp}
\end{figure}

%%%%%%%%%%%%%%%%%%%%%%%%%%%%%%%%%%%%%%%%%

\begin{table}
  \centering
   \begin{tabular}{c|cccc}
  $N$  & $p=2$ & $p=3$ & $p=4$  & $p=5$\\\hline
$5$ & $0.98$ & $1.85$ & $3.64$  & $5.07$\\
$10$ & $1.37$ & $2.10$ & $3.57$ & $5.06$\\
$20$ & $1.38$ & $2.28$ & $3.52$ & $4.77$\\
$40$ & $1.46$ & $2.42$ & $3.51$ & $4.76$\\
$80$ & $1.49$ & $2.51$ & $3.51$ & $4.63$ 
 \end{tabular}
 \hspace{.8cm}
  \begin{tabular}{c|cccc}
$N$  & $p = 2$& $p = 3$ & $p=4$ & $p=5$\\\hline
5 & 0.90 & 1.85  & 3.74 & 5.56 \\
10 & 1.17 & 2.11 & 3.39 & 5.05\\
20 & 1.34 & 2.26 & 3.41 & 4.31\\
40 & 1.44 & 2.38 & 3.41 & 4.91\\
80 & 1.45 & 2.54 & 3.46 & 4.79
\end{tabular}
  \caption{\small Numerically obtained orders of convergence of the error in the dG norm $  \dgnorm{\cdot}$  for Trefftz spaces on the left and for polynomial space on the right}
\label{tab:conv_order}

\end{table}
%%%%%%%%%%%%%%%%%%%%%%%%%%%%%%%%%%%%%%%%%%%%

\begin{table}   % \begin{tabular}{ | l | l | l | p{5cm} |}
\centering
\begin{tabular}{| l | l | l | l | l | l | l | l | l |}
\hline
$N$ &  80 & 160 & 320 & 640 & 1280 & 2560 & 5120   \\ \hline
$p = 1$     &  0.08 & 0.11 & 0.19 & 0.29 & 0.38 & 0.44 & 0.47 \\ \hline
\end{tabular}
\caption{\small Numerically obtained orders of convergence of the error in the dG norm $\dgnorm{\cdot}$ for linear elements.}\label{tab:linear_order}
\end{table}
\subsubsection{Long-time energy behaviour}

The time-space dG method that we developed is dissipative, so we expect the energy to decay over time. However, if the accuracy of the approximation is high we expect this decay to be very slow. This is indeed reflected in the numerical experiments shown in Figure~\ref{fig:energy_decaylog}, where we compute the energy
\[
E(t) = \tfrac12\|\dot u_h (t)\|^2+\tfrac12\|\nabla u_h (t)\|^2
\]
up-to time $T = 5$ with $\delta = \delta_0/4$.

\begin{figure}[hbtp]
\centering
\includegraphics[scale=0.45]{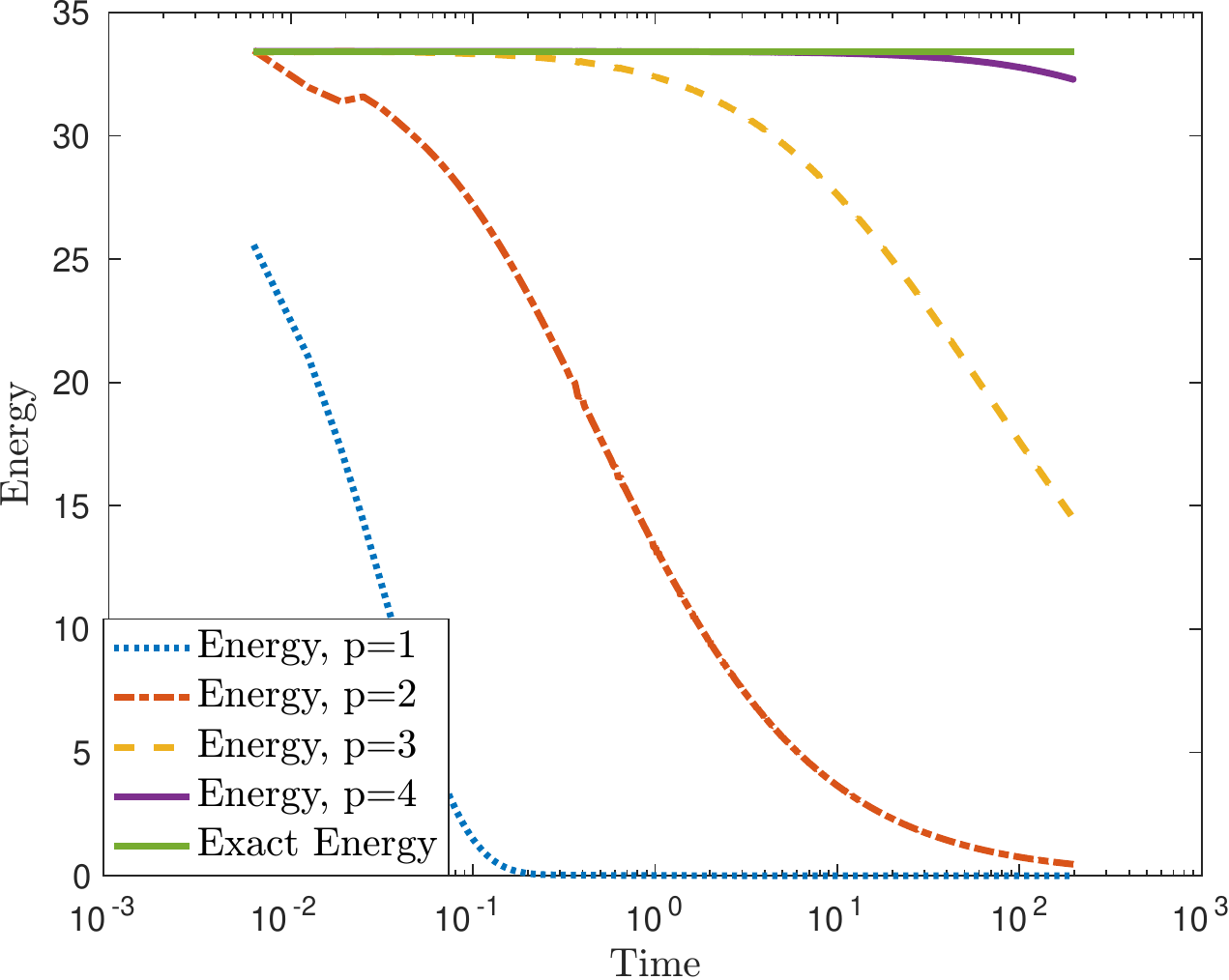}
\includegraphics[scale=0.45]{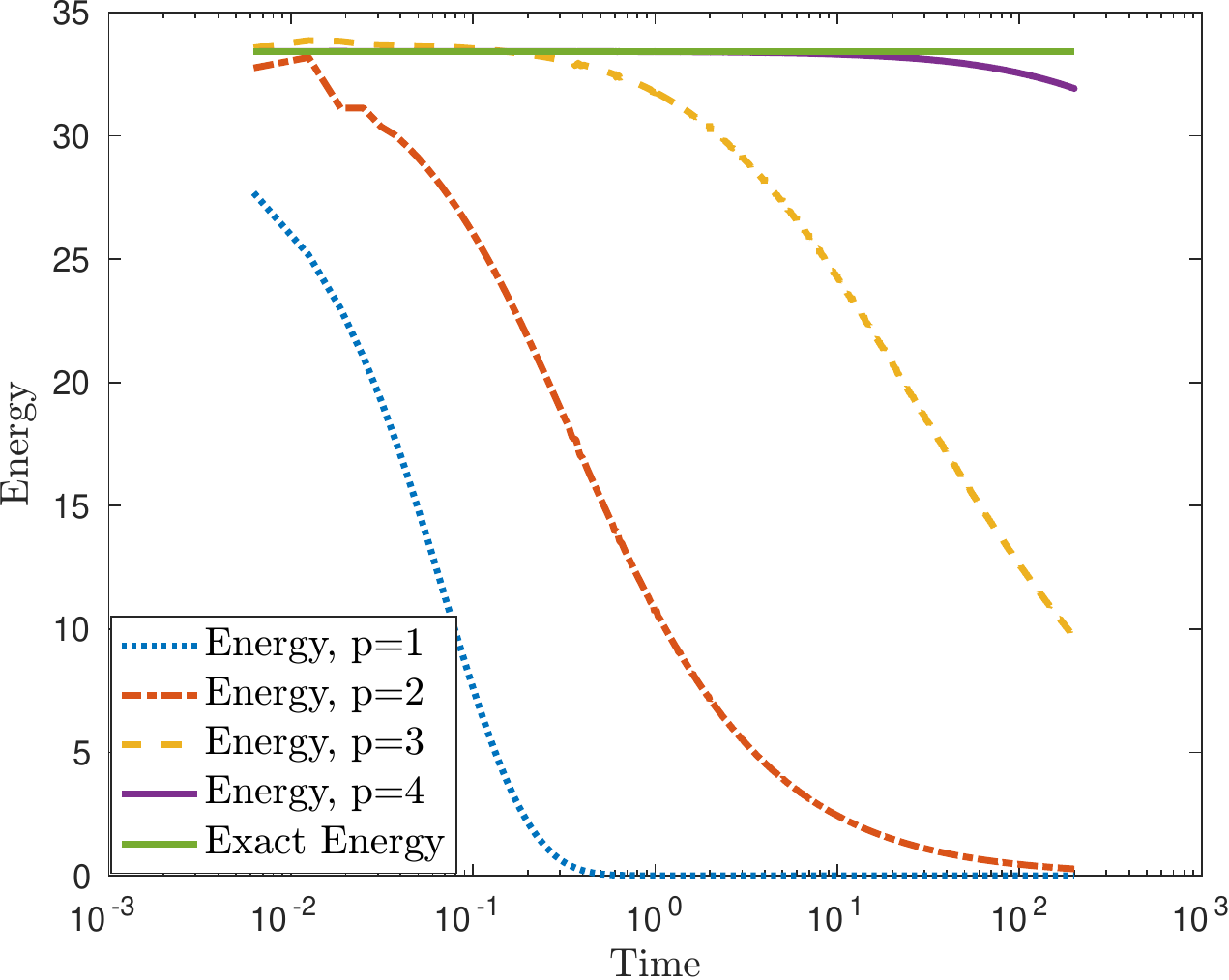}
\caption{\small Plot of the energy $E(t) = \tfrac12\|\dot u_h (t)\|^2+\tfrac12\|\nabla u_h (t)\|^2$ against time for the Trefftz spaces on the left and polynomial spaces on the right.}% Note that in the left picture, the line corresponding to $p = 4$ is not visible as it is covered by the line for the exact energy}.%produced by Energy_LongTime.m
\label{fig:energy_decaylog}
\end{figure}

\subsubsection{Waves with energy at high-frequences}

Note  that if we decrease the parameter $\delta > 0$ in the definition of the initial data \eqref{eq:init_data}, the Gaussian becomes narrower and energy at higher frequences is excited. In the following set of experiments we investigate the error while decreasing both $\delta > 0$ and  the mesh-width $h > 0$. In an ideal scenario, $h \propto \delta$ would be sufficient to obtain a constant relative error which we define as
\begin{equation}
  \label{eq:error_delta}  
\text{error}_\delta  = \left(\frac\delta2\ltwo{\dot u(\cdot,T)-\dot u_h(\cdot,T^-)}{\Omega}^2+\frac\delta2\ltwo{\nabla u(\cdot,T)-\nabla u_h(\cdot,T^-)}{\Omega}^2\right)^{1/2}.
\end{equation}
The results given in Figure~\ref{fig:error_delta} show that already for $p = 3$ the performance is good and for larger $p$ for this range of $\delta$, the error does not visibly change for different $\delta$ as long as $h \propto \delta$.

\begin{figure}[hbtp]
\centering
\includegraphics[scale=0.45]{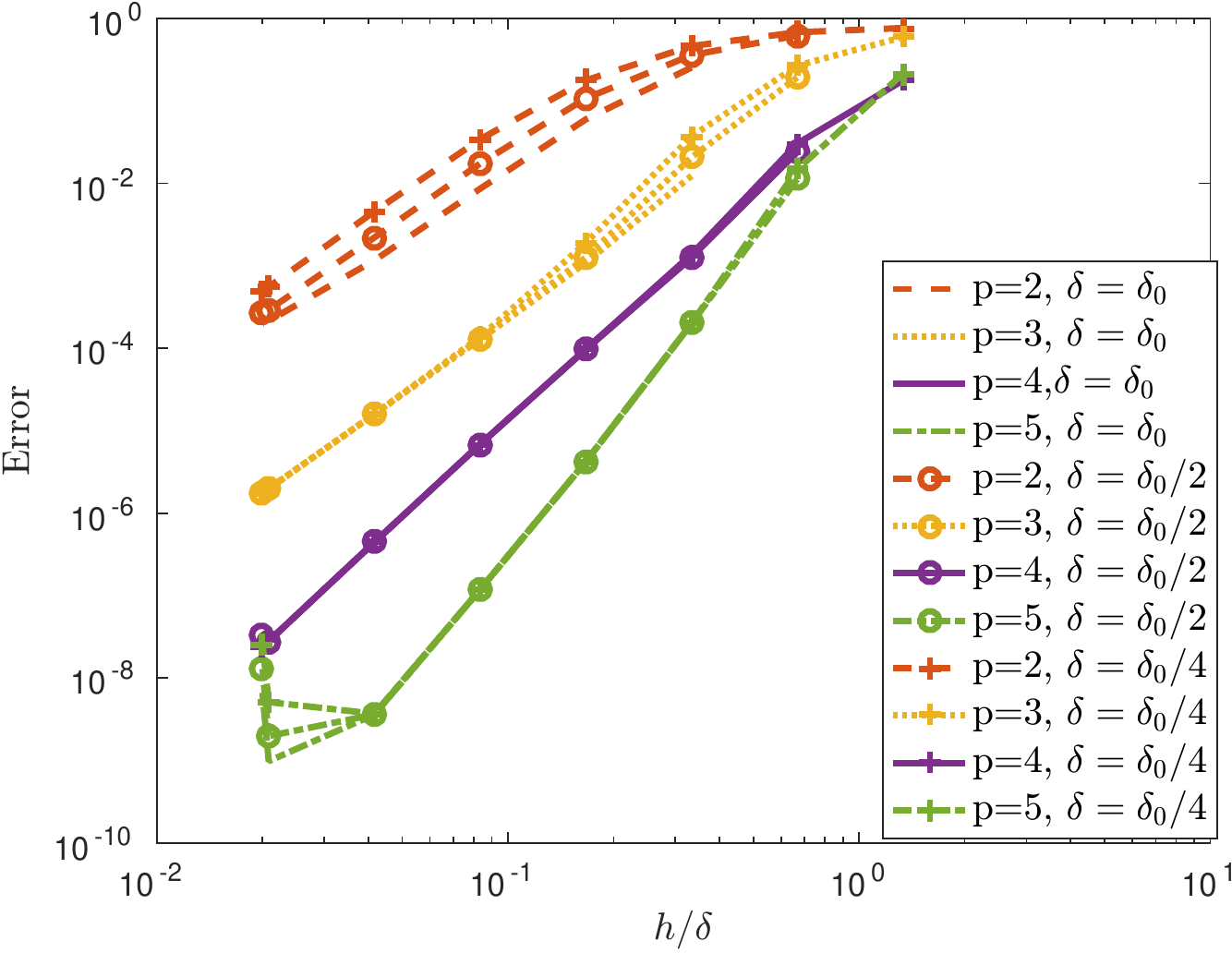} 
\caption{\small The plot of the scaled error see \eqref{eq:error_delta}, against $h/\delta$, for Trefftz space.The final time is chosen to be $T=1$. %and $K_{max}=\text{round}(1/(2\times 10^{-2}*\delta))$
\label{fig:error_delta}}
\end{figure}

\subsection{Experiments in two dimensions}

We conclude the section on numerical experiments by considering the wave equation 
\[
\ddot u -\Delta u = 0
\]
on the square $[0,1]^2 \subset \mathbb{R}^2$ with homogeneous Dirichlet boundary condition and initial data
\[
u(x,y,0) =  \sin \pi x \sin \pi y, \qquad \dot u(\cdot, 0) = 0.
\]
The analytical solution is given by
\[
u(x,y,t) = \cos(\sqrt{2}\pi t) \sin \pi x \sin \pi y.
\]
We investigate the convergence of the error in the energy norm at the final time-step
\begin{equation}
  \label{eq:2d_error}
\text{error}  = \left(\tfrac12\ltwo{\dot u(\cdot,T)-\dot u_h(\cdot,T^-)}{\Omega}^2+\tfrac12\ltwo{\nabla u(\cdot,T)-\nabla u_h(\cdot,T^-)}{\Omega}^2\right)^{1/2}.  
\end{equation}
The convergence plots are given in Figure~\ref{fig:2d_energyconv} and the computed convergence orders in Table~\ref{tab:2d_convorder}. Note that, \reply{for the weaker error notion \eqref{eq:2d_error},} we do not lose half an order of convergence as when computing the error in the discrete norm $\dgnorm{\cdot}$. The theory presented above, does not predict this behaviour. The observed rate of convergence is not, however, surprising as, unlike the dG norm, this error measure does not accumulate the errors over all time-steps.

\begin{figure}
  \centering
\includegraphics[width=.4\textwidth]{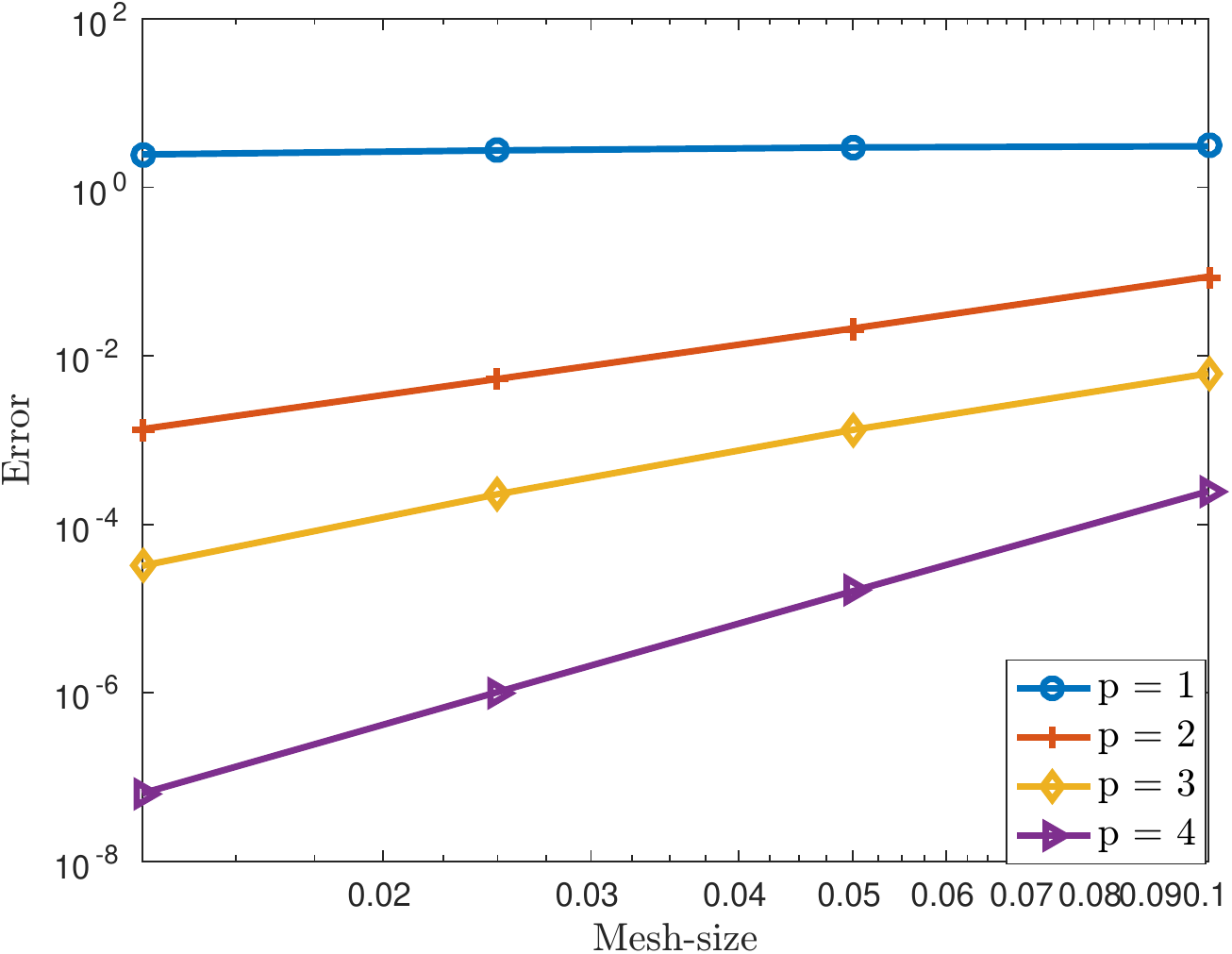}
\includegraphics[width=.4\textwidth]{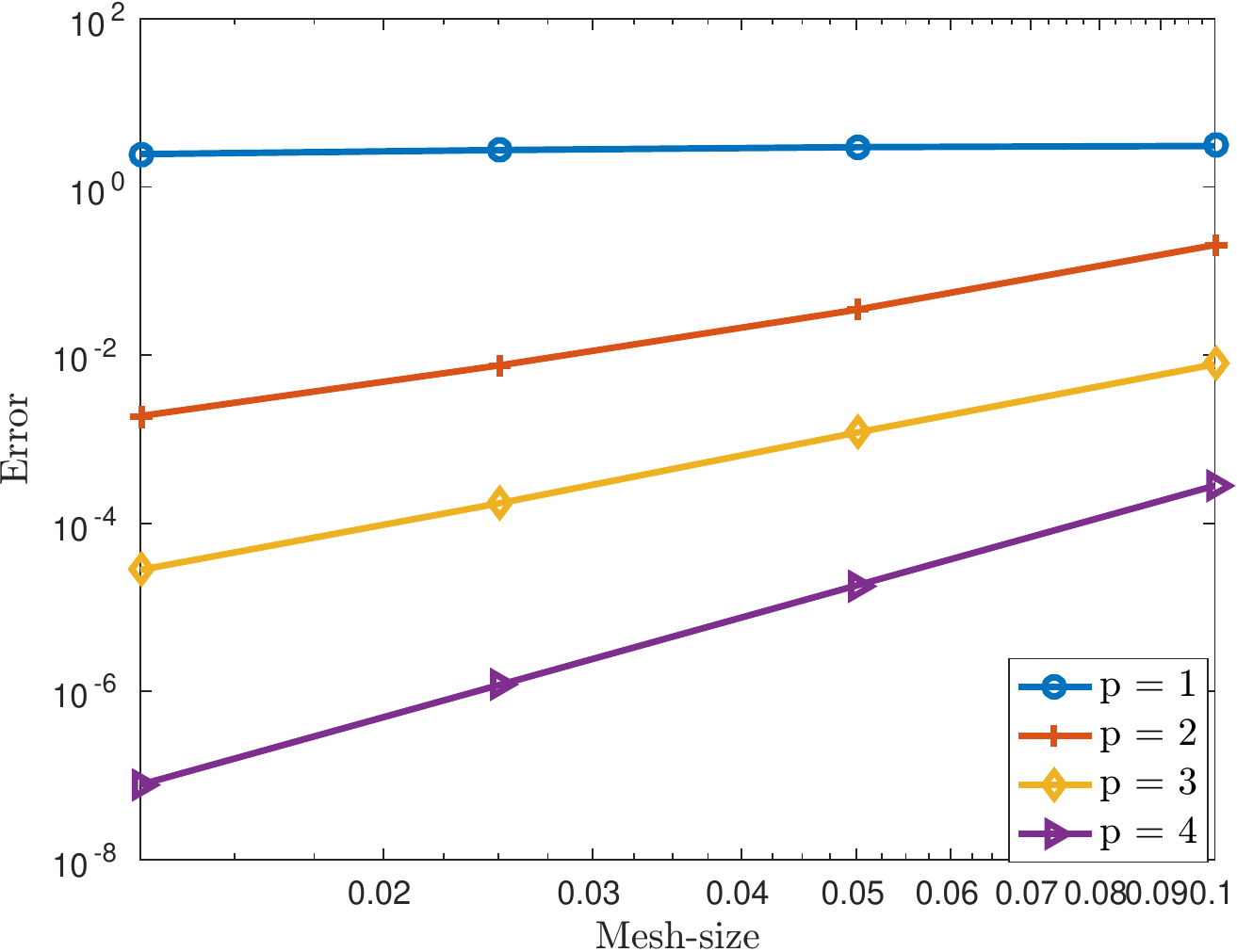}  
  \caption{\small Convergence of the error \eqref{eq:2d_error} at the final time for Trefftz basis on the left and full polynomial basis on the right.}
  \label{fig:2d_energyconv}
\end{figure}

\begin{table}[hbtp]
  \centering
   \begin{tabular}{c|cccc}
  $N$ & $p=1$ & $p=2$ & $p=3$ & $p=4$\\\hline
$10$ &$0.046$ & $2.039$ & $2.219$ & $3.935$ \\
$20$ &$0.115$ & $1.990$ & $2.539$ & $3.977$ \\
$40$ &$0.160$ & $1.979$ & $2.819$ & $3.996$ \\
 \end{tabular}
 \hspace{.8cm}
  \begin{tabular}{c|cccc}
$N$ & $p=1$ & $p = 2$& $p = 3$ & $p=4$ \\\hline
10 & 0.046 & 1.897 & 2.133 & 3.956 \\
20 & 0.115 &1.896 & 2.248 & 3.967\\
40 & 0.160 & 1.935 & 2.612 & 3.982 \\
\end{tabular}
\caption{\small Numerically obtained convergence orders of the error \eqref{eq:2d_error} at the final time for Trefftz  on the left and for the full polynomial space on the right. \label{tab:2d_convorder}}\vspace{-.5cm}
\end{table}
\FloatBarrier

%\bibliographystyle{abbrv}
%\bibliography{dgwave}

\def\cprime{$'$}

%%% Local Variables:
%%% TeX-master: "dgwave_v3"
%%% End:

\end{document}